\documentclass
[
    a4paper,
    DIV=12,
    abstract=true,
    numbers=noenddot
]
{scrartcl}

\usepackage[bf,normal]{caption}

\usepackage
{
    amsmath,
    amssymb,
    amsthm,
    array,
    authblk,
    bbm,
    bm,
    dsfont, 
    graphicx,
    mathtools,
    nicefrac,
    physics,
    tabularx,
    xcolor,
    subcaption,
    stmaryrd,
    algorithm,
    algpseudocode,
}

\usepackage[shortlabels]{enumitem}

\usepackage[T1]{fontenc}

\usepackage[pdffitwindow=false,
    plainpages=false,
    pdfpagelabels=true,
    pdfpagemode=UseOutlines,
    pdfpagelayout=SinglePage,
    bookmarks=false,
    colorlinks=true,
    hyperfootnotes=false,
    linkcolor=blue,
    urlcolor=blue!30!black,
    citecolor=green!50!black]{hyperref}

\usepackage{cleveref}

\newtheorem{theorem}{Theorem}[section]
\newtheorem{proposition}[theorem]{Proposition}
\newtheorem{lemma}[theorem]{Lemma}
\newtheorem{remark}[theorem]{Remark}
\newtheorem{corollary}[theorem]{Corollary}

\theoremstyle{definition}
\newtheorem{example}[theorem]{Example}

\newtheorem{assumption}[theorem]{Assumption}

\definecolor{darkgreen}{rgb}{0.0, 0.6, 0.6}

\newcommand{\wh}[1]{\widehat{#1}}

\newcommand{\br}[1]{\left\langle#1\right\rangle}
\newcommand{\set}[1]{\left\{#1\right\}}

\renewcommand{\norm}[1]{\left\lVert #1 \right\rVert}
\renewcommand{\abs}[1]{\left| #1 \right|}

\newcommand{\ts}{\hspace*{0.1em}} 

\newcommand\xqed[1]{\leavevmode\unskip\penalty9999 \hbox{}\nobreak\hfill \quad\hbox{#1}}
\newcommand{\exampleSymbol}{\xqed{$\triangle$}}

\newcommand{\E}{\mathbb{E}}
\newcommand{\bE}{\mathbb{E}}

\newcommand{\bN}{\mathbb{N}}

\newcommand{\bR}{\mathbb{R}}
\newcommand{\R}{\mathbb{R}}
\newcommand{\bS}{\mathbb{S}}

\newcommand{\bZ}{\mathbb{Z}}

\newcommand{\Bb}{\mathcal{B}}

\newcommand{\Ff}{\mathcal{F}}
\newcommand{\cF}{\mathcal{F}}

\newcommand{\Kk}{\mathcal{K}}

\newcommand{\Qq}{\mathcal{Q}}

\newcommand{\ini}{\xi^{m}}

\newcommand{\xtm}{X_T^{\xi^m,\mu}}

\newcommand{\xmu}{X_T^{x,\mu}}
\newcommand{\xmuhat}{X_T^{x,\hat{\mu}}}

\newcommand{\xmuhx}{X_T^{\xi^m,\hat{\mu},h}}

\newcommand{\co}{\wh{\bm{C}}_{NM}}

\newcommand{\ko}{\wh{\bm{K}}_{NM}}

\usepackage{pifont}

\DeclareMathOperator{\Var}{Var}
\DeclareMathOperator{\Law}{Law}
\DeclareMathOperator{\mspan}{span}

\allowdisplaybreaks

\DeclareCaptionLabelFormat{period}{#1~#2}
\graphicspath{{figures/}}

\begin{document}

\title{Data-driven approximation of transfer operators for mean-field stochastic differential equations}

\author[1,a]{Eirini Ioannou}
\author[1,2]{Stefan Klus}
\author[1,3,4]{Gon\c{c}alo dos Reis}
\affil[1]{Maxwell Institute for Mathematical Sciences, University of Edinburgh and Heriot--Watt University, Edinburgh, UK}
\affil[2]{Department of Mathematics, Heriot--Watt University, Edinburgh, UK}
\affil[3]{School of Mathematics, University of Edinburgh, Edinburgh, UK}
\affil[4]{Center for Mathematics and Applications (NOVA Math), Caparica, Portugal}
\affil[a]{Corresponding author: E.Ioannou-1@sms.ed.ac.uk}

\date{}%\today

\maketitle

\begin{abstract}
Mean-field stochastic differential equations, also called McKean--Vlasov equations, are the limiting equations of interacting particle systems with fully symmetric interaction potential. Such systems play an important role in a variety of fields ranging from biology and physics to sociology and economics. Global information about the behavior of complex dynamical systems can be obtained by analyzing the eigenvalues and eigenfunctions of associated transfer operators such as the Perron--Frobenius operator and the Koopman operator. In this paper, we extend transfer operator theory to McKean--Vlasov equations and show how extended dynamic mode decomposition and the Galerkin projection methodology can be used to compute finite-dimensional approximations of these operators, which allows us to compute spectral properties and thus to identify slowly evolving spatiotemporal patterns or to detect metastable sets. The results will be illustrated with the aid of several guiding examples and benchmark problems including the Cormier model, the Kuramoto model, and a three-dimensional generalization of the Kuramoto model.
\end{abstract}

\section{Introduction}

Stochastic mean-field dynamics can be understood as the limiting approximation of a random system of interacting particles that are invariant in law with respect to permutations \cite{stochastic_dynamics_2017}. Such systems have recently gained a lot of interest from various research communities in biology (e.g., neuroscience~\cite{app_neuroscience_2011}, cell biology \cite{app_cell_biology_2009}), physics (e.g., ferromagnetism and liquid crystals \cite{app_ferromagnetism_liquid_crystals_2008}, quantum systems \cite{app_quantum_2023}), as well as sociology (e.g., opinion dynamics \cite{app_opinion_dynamics_4_2017, app_opinion_dynamics_2017, app_opinion_dynamics3_2022,app_opinion_dynamics2_2023}, social hearding \cite{app_social_herding_2014}), and economics (e.g., systemic risk \cite{app_systemic_risk_2_2013, app_systemic_risk_2013}, mean-field games \cite{app_economics_2006}), see also \cite{mean_field_applications_2_2005, mean_field_applications_2010}. A random interacting particle system (IPS) has a mean-field limit if it can be described by a system of stochastic differential equations with a fully symmetric interaction term. In the limit of infinitely many interacting particles, the particles become asymptotically independent, and the equation governing the motion of each particle tends to a stochastic differential equation (SDE), called McKean--Vlasov SDE \cite{MFGS_2018}. Gaining information about the behavior of these stochastic mean-field systems using only simulation data is of great interest. Information about the global behavior of dynamical systems can be obtained by analyzing the eigenvalues and corresponding eigenfunctions of associated infinite-dimensional linear transfer operators \cite{Ko31, LaMa94, DJ99, Mezic05, KKS16}. Examples of such operators include the Perron--Frobenius operator, which describes the evolution of probability densities, and the Koopman operator, which describes the evolution of observables of the system. These transfer operators can be defined by integrals involving the transition probability kernel of the underlying system. Over the last years, many different data-driven approaches for estimating finite-dimensional approximations of these operators have been proposed. Some of the most popular methods is extended dynamic mode decomposition (EDMD)~\cite{williams2015data} and its various extensions such as kernel EDMD~\cite{WRK15, KSM19}, generator EDMD~\cite{KNPNCS20}, or deep learning-based methods such as VAMPnets~\cite{MPWN18}.

Our goal is to estimate transfer operators associated with mean-field SDEs as well as their eigenvalues and eigenfunctions. This allows us to gain insights into the global dynamics of these systems and to detect important characteristic properties including invariant distributions, metastable sets, transition rates, and implied timescales. These phenomena are understood in the following way in the context of this paper: an invariant distribution is a density of the system that remains unchanged over time, metastable sets are sets between which a dynamical transition is a rare event, a transition rate is a rate at which dynamical transitions between metastable sets happen, and timescales describe the separation of dynamics that occur in the short and long term. While transfer operators associated with conventional SDEs and numerical methods to learn these operators from data are well-understood, see, e.g., \cite{SS13, NoNu13, KKS16}, there is a need to generalize these approaches to mean-field SDEs. The global behavior of mean-field SDEs has been studied before, for instance by analyzing the stability of invariant distributions using the nonlinear Fokker--Planck equation (also known as the McKean--Vlasov PDE) \cite{soa_stationary_states_1_2001, soa_stationary_states_3_2020, bertoli_2025, bicego2025computation}. More specifically, a methodology for finding invariant distributions of mean-field models is by using Galerkin approximations (as, for example, in~\cite{bicego2025computation}), whilst a technique for analyzing their stability is by finding the Fourier modes of the interaction potential or by using Lions derivatives (e.g., as in \cite{bertoli_2025} and \cite{quentin_2024} respectively). Additionally, there is some work on parameter estimation of SDEs governing mean-field interacting particle systems with the aid of local asymptotic normality, maximum likelihood estimators, martingale estimating functions, and the Goldenshluger--Lepski method \cite{soa_par_estimation_3_2022, soa_par_estimation_4_2022, soa_par_estimation_1_2023, soa_par_estimation_2_2023}.

Transfer operators are typically linear, but the transfer operators associated to McKean--Vlasov equations are nonlinear. One way to circumvent the nonlinearity of the McKean--Vlasov PDE is to replace the law of the process within the PDE by the invariant measure, see, e.g., \cite{soa_par_estimation_3_2022}. We use a different approach that is based on the decoupled mean-field equations, where the law in the equations is decoupled from the process and then consider transfer operators associated with the decoupled SDEs. It is important to note that the decoupled and the standard McKean--Vlasov SDEs output the same solution, provided that they both start from the same initial conditions. The reason for working with the decoupled McKean--Vlasov SDE is to be able to maintain the linearity of transfer operators and also to ensure that the Markov family property holds. Decoupled  McKean--Vlasov SDEs were originally defined in order to be able to analyze the PDEs associated to these SDEs \cite{crisan2017smoothing, decoupled_pdes}.

We approximate transfer operators projected onto a space spanned by a set of basis functions, also called dictionary, using simulation data (stemming from the decoupled McKean--Vlasov SDE). This allows us to then also estimate the eigenvalues and eigenfunctions of these projected operators, which give us information about the underlying dynamics. We would like to point out that numerical results show that even if the simulations of the interacting particle system were used, we would get similar numerical estimates of the eigenvalues and eigenfunctions. However, it is unclear whether this would fail under certain conditions as we do not obtain almost sure convergence of the estimated operator to the desired operator in this case (only $L_2$ convergence), and moreover, the operators are not well-defined for the interacting particle system. The main contributions of this work are as follows:
\begin{enumerate}[leftmargin=3.5ex, itemsep=0ex, topsep=0.5ex, label=\arabic*.]
\item We show that transfer operators associated with mean-field stochastic differential equations are well-defined in the proposed framework.
\item We prove that the data-driven estimate of the projected Koopman operator tends to the exact projected Koopman operator almost surely as the number of data points goes to infinity and the step size goes to zero.
\item Furthermore, we present numerical experiments for a set of benchmark problems that illustrate the accuracy, efficiency, and versatility of our methodology.
\end{enumerate}

The remainder of the paper is structured as follows: Section~\ref{sec:mckean_vlasov_sdes_numerics} introduces the McKean--Vlasov framework. First, we define interacting particle systems, their mean-field limits, and also the decoupled McKean--Vlasov SDE. Furthermore, we show how the decoupled McKean--Vlasov SDEs can be solved numerically. In Section~\ref{sec:operators_decoupled}, we extend transfer operator theory to the mean-field setting. In Section~\ref{sec:convergence_data_limit}, we estimate the matrix representation of the projected transfer operators from data and prove that in the infinite data limit, with the step size going to zero, our method converges to the exact projected operator. Our work provides an explanation as to why and how the EDMD algorithm can be applied to mean-field SDEs. We present numerical results and illustrate that applying EDMD to the decoupled McKean--Vlasov equation allows us to gain important insights into the global behavior of the system in Section~\ref{sec:numerical_results}. Conclusions, open questions, and future work are discussed in Section~\ref{sec:conclusion}.

\section{McKean--Vlasov SDEs and numerical schemes}
\label{sec:mckean_vlasov_sdes_numerics}

In this section, we introduce interacting particle systems and the limiting systems as the number of particles goes to infinity, called McKean--Vlasov SDEs. Afterwards, we define decoupled McKean--Vlasov SDEs, which are motivated by the fact that the Markov family property is not violated for such SDEs. The decoupled McKean--Vlasov SDEs are required in order to be able to define the transfer operators in Section~\ref{sec:operators_decoupled}. Lastly, in this section we explain how artificial data can be generated from the decoupled SDEs. This is necessary as the algorithms used in Section \ref{sec:convergence_data_limit} need data as input in order to estimate transfer operators.

\subsection{Interacting particle systems and McKean--Vlasov SDEs}

In what follows, let $(\Omega, \mathcal{F},\mathbb{F},\mathbb{P})$ be a filtered probability space, where $\Omega$ is the sample space, i.e., the set of all possible outcomes ($\Omega= \mathbb{X} \subseteq \mathbb{R}^d$), $\mathcal{F}$ is the sigma algebra of events and is called the event space, $\mathbb{F}=(\mathcal{F}_t)_{t \geq 0}$ is a filtration of sigma-algebras (for which it holds that $\mathcal{F}_k \subseteq \mathcal{F}_l, \forall k \leq l$), and $\mathbb{P}$ is the probability measure. Additionally, let $(\xi^m)_{m=1,2,...,M}$ be an i.i.d.\ sequence of $\mathbb{R}^d$-valued $\mathcal{F}_0$-measurable random variables and $(W^m)_{m=1,2,...,M}$ independent $\mathbb{F}$-Brownian motion. For more details, see, e.g., \cite{chen2024IMA}.

We work with interacting particle systems in which the interaction between particles is highly symmetric. These systems are called mean-field models, where each particle is represented by an SDE. The family of mean-field interacting particle systems is of the form
\begin{equation} \label{eq:IPS}
\begin{split}
    \mathrm{d}X_t^m &= b(t,X_t^m,\mu_t^M) \ts \mathrm{d}t + \sigma(t,X_t^m, \mu_t^M) \ts \mathrm{d} W_t^m,\\
    \quad X_0^m &= \xi^m, \quad \mu_t^M = \frac{1}{M} \sum_{j=1}^M \delta_{X_t^j},
\end{split}
\end{equation}
for $ m = 1, 2, \dots, M $, where $\delta_y$ is the Dirac delta function at $y$. For each particle, $\xi^m$ denotes the initial distribution. We call $\mu_t^M$ the empirical measure of particles. Each particle has an independent noise. Note that due to the interaction coming from the empirical measure the particles are not independent.

\begin{example}[Ornstein--Uhlenbeck system]\label{example:OU}
Consider the model
\begin{equation*}
    \mathrm{d} X_t^m = a \qty(\frac{1}{M} \sum_{j=1}^M X_t^j - X_t^m) \mathrm{d}t + \sigma \mathrm{d}W_t^m, \quad X_0^m =\xi^m,
\end{equation*}
for $ m = 1,2, \dots, M $. The drift term pushes each particle toward the empirical average. We can think of the limit of this equation as $M \to \infty$ in the following way: Let $\bar{X_t} = \frac{1}{M} \sum_{j=1}^M X_t^j$, which satisfies the SDE $d \bar{X_t} = a \qty(\frac{1}{M} \sum_{j=1}^M X_t^j - \frac{1}{M} \sum_{j=1}^M X_t^j) \mathrm{d}t +  \frac{\sigma}{M} \sum_{j=1}^M d W_t^m$ and hence,
\begin{equation*}
    \bar{X_t}=\frac{1}{M} \sum_{j=1}^M X_0^j + \frac{\sigma}{M} \sum_{j=1}^M W_t^j.
\end{equation*}
Due to the law of large numbers, we have $\bar{X_t} \to \E[\xi]$ as $M \to \infty$. Hence, for $M \to \infty$ the behavior of particle $m$ should be given by
\begin{equation*}
    \mathrm{d}X_t^m = a \qty(\E[\xi] - X_t^m) \ts \mathrm{d}t + \sigma \mathrm{d}W_t^m, \quad X_0^m =\xi^m.
\end{equation*}
That is, as $M \to \infty$ the particles become asymptotically i.i.d.\ and the behavior of each particle is described by an Ornstein--Uhlenbeck process. \exampleSymbol
\end{example}

The McKean--Vlasov limit of the Ornstein--Uhlenbeck system introduced in Example \ref{example:OU} is thus given by
\begin{equation*}
    \mathrm{d} X_t = a(\E[X_t] - X_t) \ts \mathrm{d}t + \mathrm{d} W_t, \quad X_0 = \xi.
\end{equation*}
Intuitively, as $M \to \infty$, the interaction becomes weaker in the sense that the contribution of a given particle is of order $\frac{1}{M}$.

In general, the McKean--Vlasov limit of the IPS associated with \eqref{eq:IPS} is
\begin{equation} \label{eq:MVSDE}
    \mathrm{d}X_t = b(t,X_t,\mu_t) \ts \mathrm{d}t + \sigma(t,X_t,\mu_t) \ts \mathrm{d}W_t, \quad X_t = \xi, \quad \mu_t = \Law(X_t).
\end{equation}
More precisely, each particle in the $M \to \infty$ limit follows the dynamics of this SDE and is independent of the other particles.

\begin{assumption} \label{assumption:b_sigma}
Let $\Qq(\mathbb{R}^d)$ be the space of measures on $(\Omega, \mathcal{F},\mathbb{F})$ and $\Qq_2(\mathbb{R}^d)$ the space of measures $\mu \in \Qq(\mathbb{R}^d)$ such that $\int_{\Omega} \| x \|^2 \, \mathrm{d} \mu(x) < \infty$. Further, let $W_2 \colon \Qq(\mathbb{R}^d) \cross \Qq(\mathbb{R}^d) \rightarrow \mathbb{R}$ be the 2-Wasserstein distance defined by
\begin{equation*}
    W_2(\mu,\nu) = \qty( \inf_{\gamma \in \Gamma(\mu,\nu)} \int_{\Omega} d(x,y)^2 \, \mathrm{d} \gamma(x,y) )^{\frac{1}{2}},
\end{equation*}
where $\Gamma(\mu,\nu)$ is the collection of all measures with marginals $\mu$ and $\nu$. The random variable $\xi$ is in $L^p (\Omega, \cF_0)$ for all $p\geq 1$. The drift and diffusion functions are defined by $b \colon [0,T]\times \bR^d \times \Qq(\bR^d) \to \bR^d$ and $\sigma \colon [0,T]\times \bR^d \times \Qq(\bR^d) \to \bR^{d\times d}$. They satisfy the following conditions: $b$ and $\sigma$ are uniformly Lipschitz continuous in space and measure, i.e., there exists $L_b, L_\sigma\ge0$ such that for all $t \in[0,T]$ and all $x, x'\in \bR^d$ and all $\mu, \mu'\in \Qq_2(\bR^d)$ we have that
\begin{align*}
    \abs{b(t, x, \mu)-b(t, x', \mu')} &\leq L_b \big(\abs{x-x'} + W_2(\mu, \mu') \big), \\
    \abs{\sigma(t, x, \mu)-\sigma(t, x', \mu')} &\leq L_\sigma \big(\abs{x-x'} + W_2(\mu, \mu') \big).
\end{align*}
Moreover $b$ and $\sigma$ are $\nicefrac{1}{2}$-Hölder continuous in time uniformly in $x\in \bR^d$ and $\mu\in \Qq_2(\bR^d)$.
\end{assumption}

Provided that Assumption~\ref{assumption:b_sigma} is satisfied, we have well-posedness, finite moments and time-regularity for the McKean--Vlasov SDE, see~\cite[Theorem 3.2]{chen2024IMA}. We also have finite moments and a unique solution for the IPS under Assumption~\ref{assumption:b_sigma} from \cite[Proposition 2.5]{chen2024IMA}. This result also ensures that the IPS trajectories are close to the McKean--Vlasov trajectories in the $L^2$ sense, and the densities $\mu$ and $\mu^M$ are close in the Wasserstein distance, with bounds dependent on the dimension $d$ and the number of particles of the IPS $M$. With some additional assumptions on $b$ and $\sigma$, we obtain the same results but independent of the dimension $d$ as shown in \cite{poc_scalar_bel_2018}.

\subsection{Decoupled McKean--Vlasov SDEs}

The motivation for decoupled McKean--Vlasov SDEs is due to the fact that the Markov family property does not hold for McKean--Vlasov SDEs. The Markov family property implies that if we solve the SDE from every deterministic initial state, this is enough to determine the behavior of the SDE starting from a random initial state, simply by integrating over the distribution of the initial state. For McKean--Vlasov equations, this property does not hold \cite{MFGS_2018}. Intuitively, this can be explained by the fact that starting from a deterministic state $x$ implies that $\mu_0=\delta_x$, which then evolves and becomes $\mu_t$. However, this $\mu_t$ would be different if the initial state was $\mu_0 \neq \delta_x$. In other words, if we solve the McKean--Vlasov SDE with a random initial state and then condition on $X_0=x$ for some $x \in \mathbb{R}^d$, the resulting distribution is not the same as the law obtained by solving the McKean--Vlasov SDE with deterministic initial state $X_0=x$. The lack of the Markov family property would result in ill-posed transfer operators, as we would not be able to write the transfer operators as integrals. For a more detailed introduction to transfer operators and their infinitesimal generators, see Section \ref{sec:operators_decoupled}. Additionally, the operators would not be linear, hence their eigenvalues and eigenfunctions would not contain useful information. Another issue that occurs when considering McKean--Vlasov SDEs is that the infinitesimal generator of the semi-group of Koopman operators is nonlinear \cite{crisan2017smoothing}. The corresponding PDEs involve derivatives with respect to the measure as shown in \cite{lions_2010}.

In order to circumvent the problems described above, we introduce decoupled McKean--Vlasov SDEs \cite{buckdahn2014meanfieldpde,crisan2017smoothing}. The decoupled McKean--Vlasov SDE is using the $\mu_t$ density of the McKean--Vlasov SDE. However, now this density $\mu_t$ is ``independent'' of the decoupled McKean--Vlasov SDE (and the reason it is called ``decoupled'' is because the law of the coefficients is not coupled with the process that follows the SDE). More precisely,
\begin{align}\label{eq:decoupledMVSDE}
\begin{split}
    \mathrm{d}X_t^{x,\mu} &= b(t, X_t^{x,\mu},\mu_t) \ts \mathrm{d}t + \sigma(t, X_t^{x,\mu},\mu_t) \ts \mathrm{d} W_t, \\
    X_0^{x,\mu} &= x, \quad \mu = \Law(X_t).
\end{split}
\end{align}
It is important to note that $\mu_t$ is not the law of $X_t^{x,\mu}$, but the law of $X_t$ from \eqref{eq:MVSDE}, i.e., $\Law(X_t^{x,\mu}) \neq \mu = \Law(X_t)$. This means that the SDE \eqref{eq:decoupledMVSDE} is a standard SDE as its coefficients do not depend on its law. For this SDE the flow property as well as the Markov family property hold. The flow property implies that if the process were to start from time $0$ and reach time $r$, it would be the same identical process if it were to start from the process at time $s$ and move for $r-s$ amount of time. The decoupled SDE is related to the McKean--Vlasov SDE in the following way:
\begin{equation*}
    X_t^{\xi,\mu}=X_t, \quad \text{where} \quad X_0^{\xi,\mu}=\xi, X_0=\xi.
\end{equation*}
Let the transition probability density of $X_t^{x,\mu}$ be $p(\mu_0,0,t,x,z)$, which is the probability of the decoupled process to reach $z$ at time $t$ given that it started from $x$ at time $0$. The McKean--Vlasov SDE that is attached to this decoupled SDE has initial distribution $\mu_0$. Let the probability density function of the McKean--Vlasov process $X_t$ be $q(\mu_0,0,t,z)$, which is the probability that $X_t=z$  given that initial distribution at time $t=0$ was $\mu_0$. Then, we have
\begin{equation} \label{eq:densities_MVSDE}
    q(\mu_0,0,t,z) = \int_{\mathbb{R}^d} p(\mu_0,0,t,x,z) \ts \mu_0(\mathrm{d}x).
\end{equation}

We can obtain finite moments, unique solutions and the aforementioned connection between the decoupled and the McKean--Vlasov SDE from \cite[Chapter 3]{ZhangBook_2017}. It can also be proven that $X^{x,\hat{\mu}}$ is close to $X^{x,\mu}$ in the $L^2$ sense, dependent on the Wasserstein distance between $\mu$ and $\hat{\mu}$ using \cite[Theorem 3.2.4]{ZhangBook_2017}.

\subsection{Numerical scheme for the decoupled McKean--Vlasov SDE}
\label{subsec:numerical_results}

In order to simulate the decoupled McKean--Vlasov SDE \eqref{eq:decoupledMVSDE}, we would need access to $\mu$ since the Euler--Maruyama for the decoupled McKean--Vlasov numerical scheme is
\begin{equation}
\label{eq:non-implementable-EM}
    X^{x,\mu,h}_{t_{k+1}} =  X^{x,\mu,h}_{t_{k}} + b(t_k,X^{x,\mu,h}_{t_k},\mu_{t_k}) \ts h + \sigma(t_k,X^{x,\mu,h}_{t_k},\mu_{t_k}) \ts \Delta W_{k+1}, \quad X_{t_0} = x
\end{equation}
where $\Delta W_{k+1} := W_{t_{k+1}}-W_{t_k}$.
There is typically no way to access the exact $\mu$, hence we aim to estimate it by applying the Euler--Maruyama scheme to the IPS. Now if $X^{m,\tilde{h}}$ is the numerical solution of the IPS with $\tilde{M}$ particles and time step $\tilde{h}$, then $\mu_t$ can be estimated via
\begin{equation}\label{eq:mu_hat}
    \hat{\mu}_t=\mu_t^{\tilde{M},\tilde{h}}=\frac{1}{\tilde{M}} \sum_{m=1}^{\tilde{M}} \delta_{X_t^{m,\tilde{h}}}.
\end{equation}
A technical remark regarding $\mu_t^{\tilde{M},\tilde{h}}$ is that it can be understood as being simulated in a different probability space $(\tilde{\Omega}, \tilde{\cF}, \tilde{\mathbb{P}})$ and the scheme \eqref{eq:implementable-EM} runs on the reference probability space $(\Omega,\mathcal{F},\mathbb{P})$ as a $\tilde{\mathbb{P}}$-a.s. collection of random variables \cite{dosReisSmithTankov2023ISofrMVSDE}. We remark that $\hat{\mu}$ could be deterministic or random (in the sense mentioned above) dependent on how the estimation of $\mu$ is obtained. In the case of $\hat{\mu}$ being random, for ease of notation, we use the same expectation for both probability spaces.

Once $\hat{\mu}$ is computed, we generate data using the Euler--Maruyama scheme
\begin{equation}\label{eq:implementable-EM}
    X^{x,\hat{\mu},h}_{t_{k+1}} =  X^{x,\hat{\mu},h}_{t_{k}} + b(t_k,X^{x,\hat{\mu},h}_{t_k},\hat{\mu}_{t_k}) h + \sigma(t_k,X^{x,\hat{\mu},h}_{t_k},\hat{\mu}_{t_k}) \Delta W_{k+1}, \quad X_{t_0}=x.
\end{equation}

\begin{theorem} \label{theorem:decoupled_num_scheme}
Under Assumption~\ref{assumption:b_sigma} for $\mu$ as in \eqref{eq:MVSDE}, $X^{x,\mu,h}$ defined in \eqref{eq:non-implementable-EM}, $X^{x,\hat{\mu},h}$ defined in \eqref{eq:implementable-EM}, and $X^{x,\mu}$ as in \eqref{eq:decoupledMVSDE}, it holds that
\begin{equation*}
    \E\left[\abs{X_T^{x,\mu,h}-X_T^{x,\mu}}^2\right] \leq C_1 h, \quad \E\left[\abs{X_T^{x,\hat{\mu},h}-X_T^{x,\mu,h}}^2\right] \leq C_2 \sup_k W_2(\mu_{t_k},\hat{\mu}_{t_k})^2,
\end{equation*}
for some non-negative constants $C_1$, $C_2$. Additionally, for  $\mu^{\tilde{M}}_t = \frac{1}{\tilde{M}} \sum_{m=1}^{\tilde{M}} \delta_{X_t^m}$, where $X_t^m$ denotes trajectories of the interacting particle system \eqref{eq:IPS}, it holds that
\begin{equation*}
    \sup_{t\in [0,T]} \E\big[W_2(\mu_t, \mu^{\tilde{M}}_t)^2\big] \leq C_3
    \begin{Bmatrix}
        \tilde{M}^{-\nicefrac{1}{2}}, & \text{if } d<4 \\
        \tilde{M}^{-\nicefrac{1}{2}} \log \tilde{M}, & \text{if } d=4 \\
        \tilde{M}^{\frac{-2}{d}}, & \text{if } d>4
    \end{Bmatrix}
    := g(\tilde{M}),
\end{equation*}
and for $\hat{\mu}:= \mu^{\tilde{M},\tilde{h}}$ defined as in \eqref{eq:mu_hat}, it holds that
\begin{equation*}
    \sup_{t\in [0,T]}
    \bE\Big[  W_2(\mu^{\tilde{M}}_t, \mu_t^{\tilde{M},\tilde{h}})^2 \Big]
    \leq C_4 \tilde{h}
\end{equation*}
for some non-negative constants $C_3$ and $C_4$ independent of $\tilde{M}$. Hence by combining the above results, we have
\begin{equation*}
    \sup_{t\in [0,T]} \E\qty[  W_2(\mu_t, \hat{\mu}_t)^2 ] \leq g(\tilde{M})+C_4 \tilde{h}.
\end{equation*}

\begin{proof}
The proof can be found in Appendix \ref{appendix:num_scheme_results}.
\end{proof}
\end{theorem}

\section{Transfer operators for mean-field models}
\label{sec:operators_decoupled}

We now define transfer operators as well as their generators for decoupled McKean--Vlasov equations. Since these operators are infinite-dimensional objects, in order to be computed numerically, they first need to be projected onto finite-dimensional function spaces.

\subsection{Koopman operator}

Let the measure space be given by $(\mathbb{X}, \Bb, \mu_0)$, where $\mathbb{X}$ is a set on which the trajectories exist ($\mathbb{X} \subseteq \mathbb{R}^d$), $\Bb$ is a sigma-algebra on $\mathbb{X}$, and $\mu_0$ is a measure on $(\mathbb{X},\Bb)$. We would like to investigate the expected values of observables of the underlying dynamical system at lag time $T > 0$. Assume that the observable $f$ is contained in $L^{\infty}(\mathbb{X}, \Bb, \mu_0) = L^{\infty}(\mathbb{X})$, i.e., $f$ is an essentially bounded measurable function. The supremum norm is defined as $\|f\|_\infty = \operatorname{ess \  sup}_{x\in\mathbb{X}}\abs{f(x)}$, and for $f \in L^{\infty}(\mathbb{X})$ it holds that there exists a constant $C$ such that $\|f\|_\infty < C$. We define the Koopman operator with respect to the lag time $T>0$ for the decoupled process $X_T^{x,\mu}$ defined in \eqref{eq:decoupledMVSDE} as $\mathcal{K}_T \colon L^{\infty}(\mathbb{X}) \to L^{\infty}(\mathbb{X})$, with
\begin{equation} \label{eq:Koopman operator}
\begin{split}
    \Kk_T f(x)
    &
    :=\E[f(X_T^{x,\mu}) \mid X_0^{x,\mu}=x] \\
    &
    = \int_{\mathbb{X}} p(\mu_0,0,T,x,y) \ts f(y) \ts \mu_0(\mathrm{d}y),
\end{split}
\end{equation}
for $x\in\mathbb{X}$, with $p$ as in \eqref{eq:densities_MVSDE}. The Koopman operator $ \mathcal{K}_T$ is a linear operator defined on $L^{\infty}(\mathbb{X})$. For any constants $\alpha,\beta$ and any $f,g\in L^{\infty}(\mathbb{X})$ it holds that $\mathcal{K}_T(\alpha \ts f + \beta \ts g)(\cdot) = \alpha \mathcal{K}_T(f)(\cdot) + \beta \ts \mathcal{K}_T(g)(\cdot)$. This property relies only on the linearity of the expectation operator and not on the underlying stochastic process being used. Further, since $f$ is bounded, we also have $\|\Kk_T f\|_\infty\leq \|f\|_\infty$ and $\Kk_T$ is thus a bounded operator in this space.

\subsection{Perron--Frobenius operator}

The discussion here is inspired by \cite[Section 2.1]{KKS16}, but we now extend this to the decoupled mean-field case. Let the measure space $(\mathbb{X}, \Bb, \mu_0)$ as in the previous subsection. We are interested in how the SDE affects the distributions over $L^1(\mathbb{X}, \Bb, \mu_0) = L^1(\mathbb{X})$. Let $f \in L^1(\mathbb{X})$, i.e., $ \norm{f}_{L^1}=1$. For $f \in L^1(\mathbb{X})$ with $f \geq 0$ being the density of an $\mathbb{X}$-valued random variable $x \sim f$, we want to characterize the distribution of $X_T^{x,\mu}$. Suppose there exists $g$ such that $X_T^{x,\mu} \sim g$, then the mapping $f \mapsto g$ can be linearly extended to a linear operator $\mathcal{P}_T \colon L^1(\mathbb{X}) \to L^1(\mathbb{X})$. Using the transition density function $p$ as in \eqref{eq:densities_MVSDE}, we define the Perron--Frobenius operator to be
\begin{equation*}
    \mathcal{P}_T f(y) = \int_\mathbb{X}  \ts p(\mu_0,0,T,x,y) f(x) \ts \mu_0(\mathrm{d}x).
\end{equation*}

It can be proven that the Perron--Frobenius is a Markov operator. A Markov operator is an operator that preserves mass (i.e., sends the constant function of value 1 to the constant function with value 1) and preserves non-negativity \cite{operators}. To prove that the Perron--Frobenius is Markov operator, it needs to be proven that if $f\in L^1$ then $\mathcal{P}_T f \in L^1$ and that $f \geq 0$ implies $\mathcal{P}_T f \geq 0$. Suppose we have that $f \geq 0$ and since $p$ is a probability density kernel, we have $p(\mu_0,0,T,x,y) \geq 0$ for all $x,y \in \mathbb{X}$. Thus,
\begin{equation*}
    \mathcal{P}_T f(y) = \int_{\mathbb{X}}  \ts p(\mu_0,0,T,x,y) f(x) \ts \mu_0(\mathrm{d}x) \geq 0.
\end{equation*}
Now using that $p$ is a probability density kernel and hence $\int_{\mathbb{X}} p(\mu_0,0,T,x,y) \, \mu_0(dy)=1$, and also using that $ \norm{f}_{L_1}=\int_{\mathbb{X}} \abs{f(x)} \, \mu_0(\mathrm{d}x) = 1$, we have
\begin{align*}
    \int_{\mathbb{X}} \abs{\mathcal{P}_T f(y)} \ts \mu_0(\mathrm{d}y) &=\int_{\mathbb{X}} \int_{\mathbb{X}}  \ts p(\mu_0,0,T,x,y) \abs{f(x)} \ts \mu_0(\mathrm{d}x) \, \mu_0(\mathrm{d}y)
    \\ &= \int_{\mathbb{X}} \qty (\int_{\mathbb{X}} p(\mu_0,0,T,x,y) \ts \mu_0(\mathrm{d}y)) \abs{f(x)} \, \mu_0(\mathrm{d}x)
    = \int_{\mathbb{X}} \abs{f(x)}  \, \mu_0(\mathrm{d}x)
    =1.
\end{align*}
Invariant densities of the SDE are $f \in L^1(\mathbb{X})$ such that $\mathcal{P}_t f=f$ for all $t\geq 0$. Eigenfunctions associated with eigenvalues close to one correspond to the slowly evolving transients of the system and yield information about metastable sets (see, e.g., \cite{SS13, KKS16}).

\begin{lemma}\label{lemma:duality}
Defining the duality pairing $\langle \cdot, \cdot \rangle \colon L^1 \times L^\infty \to \mathbb{R}$ to be
\begin{equation*}
    \langle f, g \rangle = \int_\mathbb{X} f(x) \ts g(x) \ts \mu_0(\mathrm{d}x),
\end{equation*}
it holds that $\langle \mathcal{P}_T f, g \rangle = \langle f, \mathcal{K}_T g \rangle$.
\begin{proof}
We have
\begin{align*}
    \langle \mathcal{P}_T f,g \rangle &= \int_{\mathbb{X}} \mathcal{P}_T f(x) g(x) \, \mu_0(dx)
    \\ & = \int_{\mathbb{X}} \qty(\int_{\mathbb{X}} f(y) p(\mu_0,0,T,y,x) \,\mu_0(dy) ) g(x) \, \mu_0(dx)
    \\ & = \int_{\mathbb{X}} \int_{\mathbb{X}} f(y) p(\mu_0,0,T,y,x)   g(x) \, \mu_0(dy) \, \mu_0(dx)
    \\ & = \int_{\mathbb{X}} \int_{\mathbb{X}} f(x) p(\mu_0,0,T,x,y)   g(y) \,\mu_0(dx) \, \mu_0(dy)
    \\&= \int_{\mathbb{X}} f(x) \qty(\int_{\mathbb{X}} p(\mu_0,0,T,x,y)   g(y) \, \mu_0(dy)) \, \mu_0(dx)
    \\&=\int_{\mathbb{X}} f(x) \ts \mathcal{K}_Tg(x) \, \mu_0(dx)
    \\&=\langle f, \mathcal{K}_T g\rangle. \qedhere
\end{align*}
\end{proof}
\end{lemma}

\subsection{Infinitesimal generators of decoupled McKean--Vlasov SDEs}

Using the decoupled SDE, the approach to the Koopman operator is close to that described in \cite{KD24}. The generator of the Koopman operator follows naturally from the backward Kolmogorov equation (or the Feynman--Kac formula \cite[\textsection 5.1]{ZhangBook_2017}), while the generator of the Perron--Frobenius follows from the Fokker--Planck equation \cite[\textsection 2]{greg_book_2014}).

The infinitesimal generator of the semigroup of Koopman operators associated with the decoupled equation \eqref{eq:decoupledMVSDE} is the differential operator $\mathcal{L}$, which for some smooth map $U \colon (0,T] \cross \bR^d \to \bR$, where $(\mu_t)_t$ is understood as a fixed parameter, is defined by
\begin{equation*}
    \mathcal{L} U(t,x) = \sum_{i=1}^d b_i(t,x,\mu_t) \frac{\partial U}{ \partial x_i} + \frac{1}{2} \sum_{i,j=1}^d [\sigma(t,x,\mu_t) \sigma^\top(t,x,\mu_t)]_{i,j} \frac{\partial ^2 U}{\partial x_i \partial x_j}.
\end{equation*}
Similarly, the infinitesimal generator of the semigroup of Perron--Frobenius operators associated with the decoupled equation \eqref{eq:decoupledMVSDE} is the differential operator $\mathcal{L}^*$ defined by
\begin{equation*}
    \mathcal{L}^* U(t,x) = - \sum_{i=1}^d  \frac{\partial \qty(b_i(t,x,\mu_t) U)}{ \partial x_i} + \frac{1}{2} \sum_{i,j=1}^d  \frac{\partial ^2 \qty([\sigma(t,x,\mu_t) \sigma^\top(t,x,\mu_t)]_{i,j} U)}{\partial x_i \partial x_j}.
\end{equation*}

\subsection{Galerkin projections}

In what follows, we consider finite-dimensional approximations of the transfer operators introduced above. This subsection follows the presentation of \cite{llamazares2024data}. We note that both the Koopman and the Perron--Frobenius operators can be extended or restricted to be defined on $L^2(\mathbb{X}, \mu_0)$. Hence, from now on we assume that they are both defined on this space.

Let $\Ff := L^2(\mathbb{X}, \mu_0)$ and define the inner product on this space as;
\[ \langle f, g\rangle_{\Ff} = \int_{\mathbb{X}} f(x) \ts g(x) \, \mu_0(dx). \]

In general, $\Kk_T$ is infinite-dimensional and impossible to work with numerically. Our goal is thus to compute an approximation by projecting the operator onto a finite-dimensional subspace $\Ff_N$ of $\Ff$. This approach is called a Galerkin projection. Given a set of $N$ linearly independent functions $\Psi = \set{\psi_1, \dots, \psi_N}$, called basis functions, define
\begin{equation} \label{N_basis}
    \Ff_N := \mspan(\Psi) = \mspan\left\{\psi_1, \dots, \psi_N\right\} \subset \Ff.
\end{equation}
We denote the projection onto $\Ff_N$ by $\operatorname{Proj}_{\cF_N}$. Since $\Kk_T$ restricted to $\cF_N$, i.e., $\Kk_T|_{\cF_N}$, does not map necessarily onto $\cF_N$, we need to project it back onto $\Ff _N$. We thus define the projection of the Koopman operator $\Kk_T$ onto $\Ff_N$ by
\begin{equation} \label{eq:aux:operator-koopman-projected-onto-F_N}
    \Kk _N = \operatorname{Proj}_{\Ff_N} \mathcal{K}\big|_{\Ff_N}=\operatorname{Proj}_{\Ff_N} \Big(\mathcal{K}\big|_{\Ff_N} \Big): \Ff_N \to \Ff_N,
\end{equation}
and use the projection $\Kk_N$ as an approximation of $\Kk_T$. It is clear that
\[ \left\langle\mathcal{K}_T \psi_i, \psi_j\right\rangle_{\cF}
    =\left\langle\mathcal{K}_N \psi_i, \psi_j\right\rangle_{\cF}\]
holds based on the definition of the projected operator.

We denote operators using curly letters, and their associated matrix representations using bold letters. In essence, as described in \cite[\textsection 2.3]{{llamazares2024data}}, the finite-dimensional approximation $\mathcal{K}_N \colon \mathcal{F}_N \to \mathcal{F}_N$ is defined as
\begin{equation} \label{eq:GalerkinProjectionKoopman}
    \left(\boldsymbol{K}_N\right)^{\top} = \boldsymbol{C}_N \boldsymbol{G}_N^{-1},
\end{equation}

where for all $i,j$
\begin{equation*}
    \big[\boldsymbol{C}_N \big]_{i j}
    := \left\langle\mathcal{K}_T \psi_i, \psi_j\right\rangle_{\cF}
    \quad\text{and}\quad
    \big[\boldsymbol{G}_N\big]_{i j}:=\left\langle\psi_i, \psi_j\right\rangle_{\cF}
\end{equation*}
are the structure matrix and the Gram matrix, respectively.

The matrix representation of the projected operator $\Kk_N \colon \cF_N\to\cF_N$ is $\boldsymbol{K}_N$, and is obtained in the following way. Let $f \in \cF_N$, $f = \sum_{i=1}^N c_i \psi_i$ for some $c_i \in \mathbb{R}$, then
\begin{equation*}
    \mathcal{K}_N f = \mathcal{K}_N \sum_{i=1}^N c_i \psi_i = \sum_{i}^N \big(\boldsymbol{K}_N c\big)^{\top}_{i} \psi_i.
\end{equation*}
The Gram matrix and structure matrix for the decoupled McKean--Vlasov case are
\begin{equation} \label{eq:GNmatrix}
\begin{split}
    \big[\bm{G}_N\big]_{ij}
    = \br{\psi_i, \psi_j}_{\Ff}
    &= \int_{\mathbb{X}} \psi_i(z) \ts \psi_j(z) \ts \mu_0(\mathrm{d}z) \\
    &= \E_{x \sim \mu_0}[\psi_i(x) \ts \psi_j(x)]=\E [\psi_i(\xi) \psi_j(\xi)],
\end{split}
\end{equation}
and
\begin{equation} \label{eq:CNmatrix}
\begin{split}
     [\bm{C}_N]_{ij}
     =\br{\Kk_T \psi_i,\psi_j}_{\Ff }
     &= \int_{\mathbb{X}} \Kk_T \psi_i(z) \ts \psi_j(z) \ts \mu_0(\mathrm{d}z)
     \\
     &= \E_{x \sim \mu_0}\big[\Kk_T \psi_i(x) \ts \psi_j(x)\big]
     \\
     &= \E_{x \sim \mu_0} \left[\E\big[\psi_i(X_T^{x,\mu}) \mid X_0^{x,\mu} = x\big]\psi_j(x)\right]
     \\
     &=\E\left[\psi_i( X_T^{\xi,\mu} ) \ts \psi_j(\xi)\right],
\end{split}
\end{equation}
respectively. The aforementioned integrals cannot be explicitly defined and hence they will be estimated using Monte Carlo approximations with artificially created data. To ensure that the above quantities are well-defined, we assume that the following properties are satisfied.

\begin{assumption} \label{assumption:for_psi_kpsi}
We assume that
\begin{enumerate}[a)]  \setlength{\itemsep}{0ex}
\item the basis functions $\Psi=\set{\psi_1,\dots,\psi_N}$ are linearly independent, \label{a1}
\item the functions $\{\psi_n, \Kk_T \psi_n\}_{n=1}^N $ are continuous $\mu_0 $ almost everywhere. \label{continuous}
\end{enumerate}
\end{assumption}

The first assumption is to ensure that the Gram matrix $\mathbf{G}_N$ is invertible. The second assumption is required so that the pointwise evaluation in the Monte Carlo approximations \eqref{eq:GN_int} and \eqref{eq:CN_int} is well-defined, see Assumption 1 in \cite{llamazares2024data}.
We remark that for $\Kk_T \psi_n$ to be continuous, it is implied that $b$, $\sigma$ need to satisfy certain assumptions of regularity, and the random variable $\xi$ needs to be not ill-behaved.

\begin{assumption}[Bounds and regularity of the basis maps]
\label{assumption:psi_bounded_Lipschitz}
We assume the following:
\begin{enumerate}[a)]
\item \label{assumption:psi_bounded}
The $\psi$ maps are bounded:
\begin{itemize}
\item There exists $\gamma$ such that $\mu_0 $ almost everywhere $ \abs{\psi_n(x)} < \gamma$ for all $n=1,\dots,N$ and all $x \in \mathbb{X}$.
\item There exists $\gamma _N$ such that
$\abs{(\psi_1(x),\dots, \psi_N(x))}^2< \gamma_N \text{ for all } x \in \mathbb{X}$
(take for example $\gamma_N = 4 \gamma^2 N$).
\end{itemize}
\item \label{assumption:psi_lip}
The $\psi$ maps are (uniformly) Lipschitz continuous:
\begin{itemize}
    \item The basis functions $\psi_i$ are Lipschitz continuous with parameter $\theta_i$.
\end{itemize}
Let $\theta := \sup_{i} \theta_i$.
\end{enumerate}
\end{assumption}

This initial result establishes that the Galerkin projection of the Koopman operator is well-defined.
\begin{theorem}
\label{theo:BasicResultsOn-CN-GN-KN}

Let Assumption~\ref{assumption:for_psi_kpsi} hold, then the matrix $\boldsymbol{G}_N$ is a Gram matrix (symmetric and positive definite), and $\boldsymbol{G}_N^{-1}$ exists and is well defined.
If additionally Assumption \ref{assumption:psi_bounded_Lipschitz}\ref{assumption:psi_bounded} holds, then $\boldsymbol{G}_N, \boldsymbol{C}_N$, and $\boldsymbol{K}_N$ all have uniformly bounded entries.
\end{theorem}

\begin{proof}
The proof can be found in \cite[Section 2.3]{llamazares2024data}.
Since $\boldsymbol{G}_N$ is the Gram matrix of a linearly independent set of functions (also symmetric), we have
\begin{align*}
    \det\left(\boldsymbol{G}_N\right) \neq 0 \quad \text { a.s.~for any } N.
\end{align*}
Thus, $\bm{G}_N$ is indeed invertible. We have that $\psi_i$ are bounded for all $i$ and, using properties of expectations, we obtain $\mathcal{K}_T \psi_i$ is also bounded for all $i$. Hence, for the integrals it holds that
\begin{equation*}
    \langle \psi_i, \psi_j \rangle_{\mathcal{F}} < \infty \quad \text{and} \quad
 \langle \mathcal{K}_T \psi_i, \psi_j \rangle_{\mathcal{F}} < \infty.
\end{equation*}
Therefore, the matrices $\bm{G}_N$ and $\bm{C}_N$ are well-defined with uniformly bounded entries, which concludes that $\bm{K}_N$ also well defined with uniformly bounded entries.
\end{proof}

\subsection{The spectrum of the operator}

We seek to identify the eigenvalues and eigenfunctions of the Koopman operator associated with the underlying decoupled McKean--Vlasov dynamics. The eigenvalues closest to one represent the slowest timescales and their corresponding eigenfunctions contain information about the metastable sets of the dynamics (see, e.g., \cite{Schutte_Klus_Hartmann_2023}).

The relationship between the eigenvectors of the projected Koopman matrix representation and the eigenfunctions of the projected Koopman operator is as follows. If $v$ is an eigenvector of $\bm{K}_N \in \bR^{N\times N}$ corresponding to the eigenvalue \(\lambda \in \bR\), i.e., \(\bm{K}_N v=\lambda v\). Then the eigenfunction of $\mathcal{K}_N$ is $f(\cdot)=v^{\top} \bm{\psi}(\cdot)$ since
\begin{equation}
    \label{eq:spectrum-K-to-KN}
    (\Kk_N f)(\cdot)
    \approx
    \big(\bm{K}_N v\big)^{\top} \bm{\psi}(\cdot)
    =\lambda v^{\top} \bm{\psi}(\cdot)
    =\lambda f(\cdot),
\end{equation}
where $\bm{\psi}(\cdot) = \big(\psi_1(\cdot), \dots, \psi_N(\cdot)\big)^\top$, $\psi_i$ the basis functions. In other words, the eigenvalues of $\bm{K}_N$ approximate the Koopman eigenvalues, and the eigenvectors of $\bm{K}_N$ approximate the Koopman eigenfunctions projected onto $ \mspan(\Psi) $ \cite[Section 5.1.1]{Schutte_Klus_Hartmann_2023}.

It is clear that the procedure for finding the Koopman operator's eigenfunctions can be extended to the Perron--Frobenius operator, see explicitly Algorithm \ref{alg:EDMD}.

\section{Convergence in the infinite data limit}
\label{sec:convergence_data_limit}

In this section, we will show how extended dynamic mode decomposition (EDMD)~\cite{williams2015data} can be used to compute the matrix representations of the projected transfer operators associated with the Mckean--Vlasov SDEs from data. Additionally, we will prove that in the limit of infinitely many data points and infinitely small time steps, the data-driven matrices converge to the operator matrices.

\subsection{Data assumptions and the EDMD algorithm}

The data we use to approximate the matrix representation of $\Kk_N$ is as follows:
\begin{equation}\label{eq:data}
    \left\{\psi_n(\xi^m), \psi_n( \xmuhx)\right\}_{n, m},
    \qquad
    n=1,\ldots,N,\ m = 1, \dots, M,
\end{equation}
where $\{ \psi_n\}_{n = 1, 2, \dots, N}$ are the basis functions \eqref{N_basis}, $(\xi^m, W^m)_{m = 1, 2, \dots, M}$ are a sequence of independent copies of $(\xi,W)$, where $W$ is a Brownian motion, $\xi$ has distribution $\mu_0$, and $\xmuhx$ is the implementable numerical scheme for the decoupled McKean--Vlasov SDE defined in \eqref{eq:implementable-EM}. The data is detailed in Assumption \ref{assumption:data_num_results}\ref{assumption:data}.

The computational input to the EDMD algorithm are the (pairwise) i.i.d.\ samples \eqref{eq:data} and the data-driven approximations of $\bm{G}_N$ in \eqref{eq:GNmatrix} and $\bm{C}_N$ in \eqref{eq:CNmatrix} are given by
\begin{align}
    \label{eq:GN_int}
    \qty[\widehat{\bm{G}}_{NM}]_{ij} &:=\frac{1}{M} \sum_{m=1}^M \psi_i(\xi^m) \psi_j(\xi^m),
    \\
    \label{eq:CN_int}
        \qty[\co]_{ij}
        & :=\frac{1}{M} \sum_{m=1}^M \psi_i(\xmuhx) \psi_j(\xi^m).
\end{align}
We define the data-driven approximation of the Koopman operator as
\begin{align} \label{eq:data-driven-koopman}
    \qty[\ko]^{\top} := \co \wh{\bm{G}}_{NM}^{-1},\quad \text{where } \wh{\bm{G}}_{NM}^{-1} \text{ exists}.
\end{align}

Using the duality relation that connects the Perron--Frobenius and the Koopman operator (see Lemma \ref{lemma:duality}), we can use the results of \cite[Appendix A]{KKS16} to show that the data-driven approximation of the Perron--Frobenius operator $\mathcal{P}_T$ is
\begin{equation}
    \label{eq:data-driven-perron}
    \qty[\wh{\bm{P}}_{NM}]^{\top} := \co^{\top} \wh{\bm{G}}_{NM}^{-1}  ,\quad \text{where } \wh{\bm{G}}_{NM}^{-1} \text{ exists}.
\end{equation}
The data-driven EDMD algorithm can be found in Algorithm~\ref{alg:EDMD} (see p.\pageref{alg:EDMD}).

\begin{algorithm}
\caption{EDMD Algorithm}
\label{alg:EDMD}
\begin{algorithmic}[1]
\Require $(\xi^m, X_T^{\xi^m,\hat{\mu},h})_{m = 1, 2, \dots, M}$ and $(\psi_i)_{i = 1, 2, \dots, N}$
\State Compute the matrix $\widehat{\bm{G}}_{NM}$ via $\qty[\widehat{\bm{G}}_{NM}]_{ij} =\frac{1}{M} \sum_{m=1}^M \psi_i(\xi^m) \psi_j(\xi^m)$
\State Compute the matrix $\co$ via $\qty[\co]_{ij} = \frac{1}{M} \sum_{m=1}^M \psi_i(\xmuhx) \psi_j(\xi^m)$
\State Calculate $\qty[\ko]^{\top}= \co \wh{\bm{G}}_{NM}^{-1}$
\State Calculate $\qty[\wh{\bm{P}}_{NM}]^{\top}=  \co^{\top} \wh{\bm{G}}_{NM}^{-1} $
\State Find eigenvalues $ \lambda_l $ and eigenvectors $ v_l $ of $\ko$, $ l = 1, 2, \dots, N_{eig} $\Comment{ $N_{eig}$ is the number of eigenvalues required}
\State Find eigenvalues and eigenvectors of $\wh{\bm{P}}_{NM}$; $(\hat{\lambda}_l,\hat{v}_l)_{l = 1, 2, \dots, N_{eig}}$
\State Eigenfunctions of $\mathcal{K}_{NM}$ are calculated as $f_l(x) = \sum_{i=1}^N (v_l)_i \psi_i(x)$ for $l = 1, 2, \dots, N_{eig}$ and eigenfunctions of $\mathcal{P}_{NM}$ are calculated as $\hat{f}_l(x) = \sum_{i=1}^N (\hat{v}_l)_i \psi_i(x)$ for $l = 1, 2, \dots, N_{eig}$.
\end{algorithmic}
\end{algorithm}

We show that the approximation using the data in \eqref{eq:data} converges to the projected Koopman operator \eqref{eq:Koopman operator} associated with \eqref{eq:decoupledMVSDE} as the number of particles $M$ goes to infinity and the time step $h$ goes to zero. In other words, the study of this section is, for fixed $N$ and under appropriate conditions, to establish that (in the appropriate probabilistic sense)
\begin{equation*}
    \lim_{\substack{M\to \infty}} \widehat{\bm{G}}_{NM} = \bm{G}_N,
    \qquad
    \lim_{\substack{M\to \infty \\ h \to 0}} \co = \bm{C}_N,
    \qquad
    \lim_{\substack{M\to \infty \\ h \to 0}} \ko = \bm{K}_N,
\end{equation*}
and explain how the the eigenvalues and eigenvectors of $\ko$ relate to those of $\bm{K}_N$.

\begin{assumption}[The data regime and conditions]
\label{assumption:data_num_results}
The data \eqref{eq:data} satisfies:
\begin{enumerate}[a)]
\item \label{assumption:data}
The points $\{\xi^m\}_{m=1,\dots,M}$ are i.i.d.\ samples from the distribution $\mu_0\in \mathcal{Q}_2(\bR^d)$. The points $\{\xmuhx\}_{m=1,\dots,M}$ are samples arising from a square-integrable numerical approximation scheme with step size $h$ as in \eqref{eq:implementable-EM} with initial conditions $x=\xi^m$ and some fixed decoupling measure~$\hat \mu$.
Moreover, the family
$\{(\xi^m,\xmuhx)\}_{m=1, \dots, M}$ is pairwise i.i.d.\
\item \label{assumption:num_results}
Let $X_T^{\xi,\mu}$ be the solution to \eqref{eq:decoupledMVSDE}, let $X_T^{\xi,\mu,h}$ satisfy equation \eqref{eq:non-implementable-EM}, recall that this object is the non-implementable numerical scheme of $X_T^{\xi,\mu}$. Let $X_T^{\xi,\hat{\mu},h}$ be the corresponding numerical scheme with $\hat{\mu}$ instead of $\mu$ defined in \eqref{eq:implementable-EM}. Then, assume that we have
\begin{equation*}
    \E\left[\abs{X_T^{\xi,\mu,h}-X_T^{\xi,\mu}}^2\right] \leq C_1 h, \quad \E\left[\abs{X_T^{\xi,\hat{\mu},h}-X_T^{\xi,\mu,h}}^2\right] \leq C_2 \sup_k W_2(\mu_{t_k},\hat{\mu}_{t_k})^2.
\end{equation*}
Hence, it holds that
\begin{align*}
    \E\left[\abs{X_T^{\xi,\hat{\mu},h}-X_T^{\xi,\mu}}^2\right] &\leq 2 C_1 h + 2 C_2\sup_k W_2(\mu_{t_k},\hat{\mu}_{t_k})^2 \intertext{and} \\
    \E\left[\abs{X_T^{\xi,\hat{\mu},h}-X_T^{\xi,\mu}}\right] &\leq  \sqrt{C_1 h} + \sqrt{C_2}\sup_k W_2(\mu_{t_k},\hat{\mu}_{t_k}). \\
\end{align*}
\item \label{assumption:mu_hat_tends_to_mu} Assume that
\begin{equation*}
    \sup_{k} W_2(\mu_{t_k},\hat{\mu}_{t_k}) \to 0 \text{ as } h \to 0.
\end{equation*}
\end{enumerate}
\end{assumption}

\begin{remark} Assumptions \ref{assumption:data_num_results}
\ref{assumption:num_results} and
\ref{assumption:mu_hat_tends_to_mu} are justified due to Theorem \ref{theorem:decoupled_num_scheme} and \cite{dosReisSmithTankov2023ISofrMVSDE}. If $\hat{\mu}$ is defined as  \eqref{eq:mu_hat}, then for appropriately chosen $\tilde{h}$ and $\tilde{M}$ one can ensure the convergence of $\hat{\mu}$ to $\mu$ as $h \to 0$.
For example, if dimension of the process is less than four, the Wasserstein distance of $\mu$ and $\hat{\mu}$ squared is of order $h$ with probability $1-\delta$, where $\delta \in (0,1)$ if $\tilde{h} = \delta h$ and $\tilde{M}=\frac{1}{\delta^2 h^2}$. A proof of this result can be found in Appendix~\ref{appendix:num_scheme_results}.
\end{remark}

\subsection{Convergence results}

The convergence analysis of $ \widehat{\bm{G}}_{NM} \to \bm{G}_N $ was already carried out in \cite[Lemma 3.4]{llamazares2024data} (see also~\cite[\textsection4]{kordamezic2018OnConvEDMD}). Here, we state an extended variant of it for completeness and expand the proof.

\begin{lemma}[Convergence and invertibility of $\widehat{\bm{G}}_{NM}$]
\label{lemma:GNM-is-nice}
Let Assumption \ref{assumption:for_psi_kpsi} hold. Then $\widehat{\bm{G}}_{NM} \xrightarrow[]{a.s.} \bm{G}_{N}$ as $M \rightarrow \infty$. Moreover, there exists a $\mathbb{Z}^+$-valued random variable $m_0$ such that $\widehat{\bm{G}}_{NM}$ is a.s.-invertible for any $M\geq m_0$. Provided that the additionally Assumption~\ref{assumption:psi_bounded_Lipschitz} holds, we have that $\widehat{\bm{G}}_{NM} \xrightarrow[]{L^2(\Omega)} \bm{G}_{N}$ as $M\to\infty$ and there exists a number $m_0^*\in \bZ^+$ such that for any $M\geq m_0^*$ the matrix $\widehat{\bm{G}}_{NM}$ is invertible (in $L^2(\Omega)$).
\end{lemma}

\begin{proof}
This result follows partly from \cite[Lemma 3.4]{llamazares2024data}. Since the family $\{\xi^m\}_m$ is assumed i.i.d.\ and square-integrable, then for any $i,j$ the family of random variable $Y_m:=\psi_i(\xi^m) \psi_j(\xi^m)$ is i.i.d.\ and satisfies $\bE\big[\abs{Y_m}\big]<\infty$ (due to the square-integrability of $\psi_k$), then, in turn, via the strong Law of Large Numbers (see Lemma \ref{obs_l2_wlln}) we deduce that, for all $i, j \in \{1, \dots, N\}$, when $M \to \infty$
\begin{align*}
    \left[\widehat{\boldsymbol{G}}_{N M}\right]_{i j}
    =\frac{1}{M} \sum_{m=1}^M \psi_i(\xi^m) \ts \psi_j(\xi^m)
    \xrightarrow[]{a.s.}
    \bE[ \psi_i(\xi) \ts \psi_j(\xi)]
    =
    \int \psi_i(y) \ts \psi_j(y) \ts \mu_0(\mathrm{d}y) =\left[\boldsymbol{G}_N\right]_{i j},
\end{align*}
where $\xi\sim \mu_0$.

From Theorem \ref{theo:BasicResultsOn-CN-GN-KN} we have that $\boldsymbol{G}_N$ is well-defined and that $\boldsymbol{G}_N$ is invertible with $\det\left(\boldsymbol{G}_N\right) \neq 0$. Using that the determinant operator is a continuous map we have
\begin{equation*}
    \det\Big(\widehat{\boldsymbol{G}}_{N M}\Big)
    \xrightarrow[]{a.s.}
    \det\Big(\boldsymbol{G}_N\Big) \neq 0 \quad \text { as } M \to \infty.
\end{equation*}
From the above a.s.-convergence we conclude that for any given $\varepsilon>0$ the existence of a random variable $m_0 \in \mathbb{N}$ such that
\begin{equation*}
    \abs{\det\Big(\widehat{\boldsymbol{G}}_{N M}\Big) - \det\Big(\boldsymbol{G}_N\Big)} < \varepsilon \quad \text { a.s. when } M \geq m_0.
\end{equation*}
Hence, a.s.~$\det(\widehat{\boldsymbol{G}}_{N M}) \neq 0$ for any $M \geq m_0$ ensuring the invertibility of $\widehat{\boldsymbol{G}}_{N M}$. This proves the result.

The second statement follows from the first using the boundedness of the basis functions $\psi$ in Assumption~\ref{assumption:psi_bounded_Lipschitz} (and employing the dominated convergence theorem), yielding for all $i, j \in\{1, \dots, N\}$
\begin{equation*}
    \bE\Big[ \abs{
    [\widehat{\boldsymbol{G}}_{N M}]_{i j}
    -
    [\boldsymbol{G}_N]_{i j}
    }^2 \Big]
    \xrightarrow[]{}0,\qquad \text{as\ }M \to \infty.
\end{equation*}
The invertibility result follows in a similar way (using the continuity of the determinant operator and the convergence in $L^2(\Omega)$).
\end{proof}

\begin{lemma}[ $L^2(\Omega)$-convergence of $\co$]\label{lem_conv_co}Let Assumptions \ref{assumption:for_psi_kpsi}, \ref{assumption:psi_bounded_Lipschitz} and \ref{assumption:data_num_results} hold. Then, for fixed $N$, the matrices $\co$ defined in \eqref{eq:GN_int} and $\bm{C}_N$ defined in \eqref{eq:CNmatrix} satisfy
\begin{equation*}
    \co \xrightarrow{L^2(\Omega)} \bm{C}_N \quad \text{as } M \to \infty \text{ and as } h \to 0.
\end{equation*}
\end{lemma}

\begin{proof}
Fix $N$ and recall $\co$ defined in \eqref{eq:CN_int}. Fix some $i,j \in \{1, \dots, N\}$ and without loss of generality, we prove the convergence of $[\co]_{ij}$ to $[\bm{C}_N]_{ij}=\E\left[\psi_i( X^{\xi,\mu}_T )\psi_j(\xi)\right]$. Convergence for the whole matrix then follows.

Consider the decomposition of $[\co]_{ij}$ given by, where $\xtm$ is the solution to \eqref{eq:decoupledMVSDE} starting from position $\xi^m$,
\begin{align*}
    \big[\co\big]_{ij} &= \frac{1}{M} \sum_{m=1}^M \psi_i(\xmuhx)\psi_j(\ini) \\
    &=\underbrace{\frac{1}{M} \sum_{m=1}^M \left(\psi_i(\xmuhx) - \psi_i(\xtm)\right) \psi_j(\ini)}_\text{$I_1$} + \underbrace{\frac{1}{M} \sum_{m=1}^M \psi_i(\xtm) \psi_j(\ini)}_\text{$I_2$},
\end{align*}
where $ X_0^{m,h,M} = \ini \text{ for } m = 1, 2, \dots, M $. Using the triangle inequality, we have
\begin{equation*}
    \abs{ \qty[\co - \bm{C}_N]_{ij} } \leq \abs{I_1} + \abs{I_2-\qty[\bm{C}_N]_{ij}}.
\end{equation*}
Taking the expectation on both sides and using Jensen's inequality, we have
\begin{equation*}
    \E \left[\abs{ \qty[\co - \bm{C}_N]_{ij} }^2 \right] \leq 2\E\left[\abs{I_1}^2\right] + 2\E\left[\abs{I_2-\qty[\bm{C}_N]_{ij}}^2\right].
\end{equation*}

We consider each term of the right-hand side of the above inequality separately. Using Jensen's inequality, the boundedness of the basis functions and their Lipschitz property, we obtain
\begin{align*}
    \E\left[\abs{I_1}^2\right]
    &
    \leq \frac{1}{M} \sum_{m=1}^M \E \left[ \abs{ \psi_i(\xmuhx) -\psi_i(\xtm)}^2 \abs{\psi_j(\ini)}^2  \right]
    \\
    &\leq \frac{\gamma^2 \theta^2}{M} \sum_{m=1}^M \E \left[ \abs{ \xmuhx -\xtm }^2  \right]
    \\
    &\leq \frac{\gamma^2 \theta^2 }{M} M (2 C_1 h + 2 C_2 \sup_k W_2(\mu_{t_k},\hat{\mu}_{t_k})^2) \to 0 \quad \text{as } h \to 0,
\end{align*}
where the last convergence follows from Assumption \ref{assumption:data_num_results}.

For the second term in the right-hand side, we use Lemma \ref{obs_l2_wlln} as the results concern only a family of i.i.d.\ random variables. We can apply this result because $\{\psi_i(\xtm)\psi_j(\ini)\}_{m = 1, 2, \dots, M}$ are independent with mean $\E[\psi_i(\xtm)\psi_j(\ini)]$ and (uniformly) bounded covariance since $\psi_i$ and $\psi_j$ are bounded. Concretely,
\begin{align*}
    \E\Big[\,|I_2-\qty[\bm{C}_N]_{ij}|^2\Big]
     &
     =\E\left[ \abs{
               \frac{1}{M} \sum_{m=1}^M \psi_i(\xtm)\psi_j(\ini)
                -\E\big[\psi_i(\xtm)\psi_j(\ini)\big]
               }^2 \right]
     \\
     &
     \leq
     \frac1M \Var\Big(\psi_i(X_T^{\xi,\mu})\psi_j(\xi)\Big) \to 0 \text{ as } M \to \infty.
\end{align*}

Therefore, gathering both estimates, we conclude that
\begin{equation*}
    \E \left[\abs{ \qty[\co - \bm{C}_N]_{ij} }^2 \right]  \to 0 \quad \text{ as } M \to \infty \text{ and } h \to 0,
\end{equation*}
and thus, due to the matrices being finite dimensional,
\begin{equation*}
    \co \xrightarrow{L^2( \Omega )} \bm{C}_N \text{ as } M \to \infty \text{ and } h \to 0. \qedhere
\end{equation*}
\end{proof}

\begin{lemma}[Almost sure convergence of $\co$] \label{lemma:convergence_as_co}
Let $b$ and $\sigma$ satisfy Assumptions \ref{assumption:b_sigma}. Let $\psi_i$ satisfy Assumptions \ref{assumption:for_psi_kpsi} and \ref{assumption:psi_bounded_Lipschitz}. Additionally, assume that Assumption \ref{assumption:data_num_results} holds. Then, for fixed $N$, $\co$ defined in \eqref{eq:CN_int} and $\bm{C}_N$ defined in \eqref{eq:CNmatrix} satisfy
\begin{equation*}
    \co \xrightarrow{a.s.} \bm{C}_N + \bm{E}(h) \quad \text{as } M \to \infty,
\end{equation*}
where $\bm{E}(h)$ is some deterministic matrix in $\mathbb{R}^{N \times N}$ dependent on $h$ such that
\begin{equation*}
    \bm{E}(h) \rightarrow \bm{0}_{N \times N} \quad \text{as } h \to 0.
\end{equation*}
The matrix $\bm{E}(h)$ is explicitly given in the proof.

\begin{proof}
Fix some $i,j \in \{1, \dots, N\}$ and without loss of generality, we prove almost sure convergence of $[\co]_{ij}$ to $[\bm{C}_N]_{ij}=\E\left[\psi_i( X^{\xi,\mu}_T )\psi_j(\xi)\right]$. Consider the decomposition of $[\co]_{ij}$ given by
\begin{align*}
    \big[\co\big]_{ij} &= \frac{1}{M} \sum_{m=1}^M \psi_i(\xmuhx)\psi_j(\ini) \\
    &=\underbrace{\frac{1}{M} \sum_{m=1}^M \left(\psi_i(\xmuhx) - \psi_i(\xtm)\right) \psi_j(\ini)}_\text{$I_1$} + \underbrace{\frac{1}{M} \sum_{m=1}^M \psi_i(\xtm) \psi_j(\ini)}_\text{$I_2$},
\end{align*}
where $ X_0^{m,h,M} = \ini \text{ for } m  = 1, 2, \dots, M $. Using the triangle inequality, we have
\begin{equation*}
    \abs{\qty[\co - \bm{C}_N]_{ij}} \leq \abs{I_1} + \abs{I_2-\qty[\bm{C}_N]_{ij}}.
\end{equation*}
We would like to prove that
\begin{equation} \label{eq:i1_and_i2_cn_as}
    \abs{I_1} \xrightarrow[]{a.s.} \bm{E}(h), \quad \abs{I_2-\qty[\bm{C}_N]_{ij}} \xrightarrow[]{a.s.} 0 \text { as } M \to \infty.
\end{equation}
In order to prove the first part of \eqref{eq:i1_and_i2_cn_as}, we use the strong law of large numbers is applied, which we can apply because $(\xmuhx,\xtm)_{m = 1, 2, \dots, M}$ are pairwise i.i.d.\ and $\psi_i$ are bounded.
Thus,
\begin{equation*}
    \frac{1}{M} \sum_{m=1}^M \left(\psi_i(\xmuhx) - \psi_i(\xtm)\right) \psi_j(\ini) \xrightarrow[]{a.s.} \E\qty[\left(\psi_i(X_T^{\xi,\hat{\mu},h}) - \psi_i(X_T^{\xi,\mu})\right) \psi_j(\xi)] \text{ as } M \to \infty.
\end{equation*}
Define $\bm{E}(h)$ as $\qty[\bm{E}(h)]_{ij}:=\E\qty[\left(\psi_i(X_T^{\xi,\hat{\mu},h}) - \psi_i(X_T^{\xi,\mu})\right) \psi_j(\xi)]$.

For the second part of \eqref{eq:i1_and_i2_cn_as}, we use strong law of large numbers. We can apply this result because $\{\psi_i(\xtm)\psi_j(\ini)\}_{m = 1, 2, \dots, M}$ are i.i.d.\ with mean $\E[\psi_i(\xtm) \ts \psi_j(\ini)]$ and (uniformly) bounded covariance since $\psi_i$ and $\psi_j$ are bounded. More precisely,
\begin{equation*}
    \abs{I_2-\qty[\bm{C}_N]_{ij}}
     = \frac{1}{M} \sum_{m=1}^M \psi_i(\xtm)\psi_j(\ini)
                -\E\big[\psi_i(\xtm)\psi_j(\ini)\big] \xrightarrow[]{a.s.} 0 \text{ as } M \to \infty.
\end{equation*}
Therefore, gathering both estimates we conclude
\begin{equation*}
    \abs{[\co - \bm{C}_N]_{ij}}  \xrightarrow[]{a.s.} \qty[\bm{E}(h)]_{ij} \quad \text{ as } M \to \infty,
\end{equation*}
and thus
\begin{equation*}
    \co \xrightarrow{a.s.} \bm{C}_N + \bm{E}(h) \text{ as } M \to \infty.
\end{equation*}

It remains to show that
\begin{equation*}
    \left(\psi_i(X_T^{\xi,\hat{\mu},h}) - \psi_i(X_T^{\xi,\mu})\right) \psi_j(\xi) \xrightarrow[]{L^1} 0 \text{ as } h \to 0.
\end{equation*}
Rewriting the expectation, we would like to show convergence to $0$ in the following way:
\begin{align*}
    \E\qty[\left(\psi_i(X_T^{\xi,\hat{\mu},h}) - \psi_i(X_T^{\xi,\mu})\right) \psi_j(\xi)] & \leq \E\qty[ \theta_i \abs{X_T^{\xi,\hat{\mu},h} - X_T^{\xi,\mu}} ] \gamma
    \\ & \leq \theta_i  \E\qty[ \abs{X_T^{\xi,\hat{\mu},h} - X_T^{\xi,\mu,h} +  X_T^{\xi,\mu,h} - X_T^{\xi,\mu}} ] \gamma
    \\ & \leq \theta_i  \E\qty[ \abs{X_T^{\xi,\hat{\mu},h} - X_T^{\xi,\mu,h}}]\gamma + \theta_i \E\qty[\abs{X_T^{\xi,\mu,h} - X_T^{\xi,\mu}}] \gamma.
\end{align*}
By Assumption \ref{assumption:data_num_results} \ref{assumption:num_results} and \ref{assumption:mu_hat_tends_to_mu} we have the following statements
\begin{align*}
    \E\left[\abs{ X_T^{\xi,\mu,h} - X_T^{\xi,\mu}}\right] &\leq \sqrt{C_1} \sqrt{h}, \\
    \E\left[\abs{X_T^{\xi,\hat{\mu},h} - X_T^{\xi,\mu,h}}\right] &\leq \sqrt{C_2} \sup_{k = 1, 2, \dots, K} W_2 (\mu_{t_k},\hat{\mu}_{t_k}) \to 0 \text{ as } h \to 0,
\end{align*}
from which we can conclude that
\begin{equation*}
    \qty[\bm{E}(h)]_{ij} = \E\qty[\left(\psi_i(X_T^{\xi,\hat{\mu},h}) - \psi_i(X_T^{\xi,\mu})\right) \psi_j(\xi)] \to 0 \text{ as } h \to 0,
\end{equation*}
and hence we get
\begin{equation*}
   \bm{E}(h) \rightarrow \bm{0}_{N \times N} \quad \text{as } h \to 0. \qedhere
\end{equation*}
\end{proof}
\end{lemma}

Taking the infinite-data limit of $\co$ results in almost sure convergence to a perturbation of the matrix $\bm{C}_N$ with the perturbation coming from the errors produced by $W_2(\mu,\hat{\mu})$ and the time step $h$. Hence taking the limit $h \to 0$ after the infinite-data limit results in the convergence to the exact desired matrix $\bm{C}_N$. Note that almost sure convergence can be obtained here, due to the trajectories being independent and identically distributed.

\begin{theorem}[Convergence of $\ko$] \label{lem_conv_ko}
If Assumptions \ref{assumption:for_psi_kpsi}, \ref{assumption:psi_bounded_Lipschitz}, and \ref{assumption:data_num_results} hold, then
\begin{equation*}
    \qty[\ko]^{\top} \xrightarrow{L^2( \Omega )} \bm{C}_N \bm{G}_N^{-1} \quad \text{as } M \to \infty \text{ and } h \to 0.
\end{equation*}
\end{theorem}

\begin{proof}
This result follows from Lemma \ref{lemma:GNM-is-nice} and Lemma~\ref{lem_conv_co}. We have shown that $\co\to \bm{C}_N$ and $\wh{\bm{G}}_{NM}^{-1}\to  \bm{G}_N^{-1}$ in $L^2(\Omega)$ as $M\to \infty, h\to 0$.
Additionally, Assumption~\ref{assumption:psi_bounded_Lipschitz} ensures that each sequence is bounded uniformly in $M,h$ (at fixed $N$) and thus the product of the sequences also converges in $L^2(\Omega)$.

Assume that for fixed $N$ the sequences $\co \to \bm{C}_N$ and $\widehat{\bm{G}}_{NM}^{-1} \to \bm{G}_N^{-1}$ converge in $L^2(\Omega)$ and that both sequences are uniformly bounded, then the product also converges in $L^2(\Omega)$.
To see this, we write
\begin{align*}
    \co \widehat{\bm{G}}_{NM}^{-1} - \bm{C}_N \bm{G}_N^{-1} = (\co & - \bm{C}_N)(\widehat{\bm{G}}_{NM}^{-1} - \bm{G}_N^{-1})
\\
&
+ (\co - \bm{C}_N)\bm{G}_N^{-1}
+ \bm{C}_N (\widehat{\bm{G}}_{NM}^{-1} - \bm{G}_N^{-1}).
\end{align*}
Each term on the right-hand side converges to zero in $L^2(\Omega)$.
By uniform boundedness and $L^2$ convergence, we have
\begin{equation*}
    \norm{ (\co - \bm{C}_N)(\widehat{\bm{G}}_{NM}^{-1} - \bm{G}_N^{-1}) }_{L^2(\Omega)}
    \leq
    \norm{\co - \bm{C}_N}_{\infty} \cdot \norm{\widehat{\bm{G}}_{NM}^{-1} - \bm{G}_N^{-1}}_{L^2(\Omega)} \to 0.
\end{equation*}

The second and third terms converge similarly due to the (uniform) boundedness of $\bm{G}_N^{-1}$ and $\bm{C}_N$ and the $L^2(\Omega)$ convergence of each sequence.
Therefore,
\begin{equation*}
    \qty[\ko]^\top =
    \co \widehat{\bm{G}}_{NM}^{-1} \xrightarrow{L^2(\Omega)} \bm{C}_N \bm{G}_N^{-1}\ \text{ as } \ M\to\infty \text{ and } h\to 0. \qedhere
\end{equation*}
\end{proof}

If $h$ is fixed, then the convergence in $M$ can be established in almost sure sense but with a nuance. The infinite-data limit convergence in \eqref{eq:data} is to the numerical discretization underlying the approximation of the equation's solution process $\psi_n(X_T^{\xi,\hat\mu,h})$.

\begin{theorem}[Almost sure convergence of $\ko$]
\label{lem_conv_as_ko}

Let $b$ and $\sigma$ satisfy Assumptions \ref{assumption:b_sigma}. Let $\psi_i$ satisfy Assumptions \ref{assumption:for_psi_kpsi} and \ref{assumption:psi_bounded_Lipschitz}. Additionally, assume that Assumption~\ref{assumption:data_num_results} holds. Then
\begin{equation*}
    \qty[\ko]^\top \xrightarrow{a.s.} (\bm{C}_N + \bm{E}(h)) \bm{G}_N^{-1} \quad \text{as } M \to \infty,
\end{equation*}
and
\begin{equation*}
   (\bm{C}_N + \bm{E}(h)) \bm{G}_N^{-1} \rightarrow \bm{C}_N \bm{G}_N^{-1} \quad \text{as } h \to 0.
\end{equation*}
$\bm{E}(h)$ is defined as in Lemma \ref{lemma:convergence_as_co}.
\end{theorem}

\begin{proof}
This result follows from Lemma~\ref{lemma:GNM-is-nice} and Lemma \ref{lemma:convergence_as_co}. The proof follows the arguments from \cite[Theorem 6.3]{llamazares2024data}. For $M>m_0$, $m_0$ chosen as in Lemma \ref{lemma:GNM-is-nice}, we have that $\wh{\bm{G}}_{NM}$ is invertible and by properties of almost sure convergence get the desired result.
\end{proof}

It can be shown that we obtain eigenvalue convergence of the estimated projected Koopman matrix to the eigenvalues of the exact projected Koopman operator. The theorem and its proof follow from \cite[Theorem 6.1]{llamazares2024data} and from the continuity of the spectrum map of matrices \cite{stewart1990matrix, kato1995perturbation}.

\begin{remark}
It is important to note that if we were to use data simulated from the interacting particle system to estimate the matrices $\wh{\bm{C}}_{NM}, \wh{\bm{G}}_{NM}$ and $\wh{\bm{K}}_{NM}$, then we would be able to obtain $L^2(\Omega)$ convergence to the desired projected matrices $\bm{C}_N, \bm{G}_N$ and $\bm{K}_N$ respectively. However, almost sure convergence would not be obtained. This restriction is due to the dependence of the particle trajectories to each other, which forbids us from using the strong law of large numbers, and hence does not permit us to prove almost sure convergence. The lack of almost sure convergence prevents us from constructing a theorem corresponding to \cite[Theorem 6.1]{llamazares2024data} for the interacting particle system case. However, even though we cannot mirror the aforementioned theorem of convergence for the case of IPS data, something can be said about the eigenvalue bounds, and hence the convergence of the eigenvalues in the case of $\bm{K}_N$ being a diagonalizable matrix (using the Bauer--Fike Theorem~\cite{BAUERFIKE1960}).

In conclusion, in the case of using interacting particle system data (with $\bm{K}_N$ not diagonalizable), we are not able to infer much information about the eigenvalues and eigenfunctions of the projected operator $\bm{K}_N$ from the calculated matrix $\wh{\bm{K}}_{NM}$.
\end{remark}

\section{Numerical results}
\label{sec:numerical_results}

In this section we show how EDMD can be applied to mean-field SDEs, demonstrating the efficacy and accuracy of the proposed approach. We compute the eigenfunctions and eigenvalues of Koopman and Perron--Frobenius operators associated with three benchmark problems.

\subsection{Cormier model}

The Cormier model, which can be found in \cite{quentin_2024}, exhibits interesting behavior as it can, depending on its parameters, have multiple invariant distributions and the associated linear process (i.e., the system when one of the invariant distributions is plugged in the expectation) as defined in \cite{quentin_2024} is the Ornstein--Uhlenbeck process. The McKean--Vlasov version of this SDE is given by
\begin{equation*}
    \mathrm{d}X_t = -X_t \ts \mathrm{d}t + J \ts \E[\cos(X_t)] \ts \mathrm{d}t + \sqrt{2} \ts \mathrm{d}W_t, \quad X_0 \sim \operatorname{Unif}[a, b],
\end{equation*}
where $J$, $a$, and $b$ are parameters. The invariant distributions of this SDE are of the form $\mathcal{N}(\alpha,1)$, where $\alpha$ is such that $\frac{\sqrt{e}}{J} \alpha = \cos(\alpha)$. The stability of these invariant distributions and criteria determining which distributions are stable can be found in \cite[Section 2.2]{quentin_2024}. It is in particular shown that an invariant distribution $\mathcal{N}(\alpha,1)$ is stable if and only if $\alpha \tan(\alpha)>-1$. This can also be verified using numerical experiments, as regardless of the initial conditions, the particles never converge to one of the unstable invariant distributions.

We simulate 500\ts000 particles using the decoupled numerical scheme, with step size $h = 0.1$, $J=14$, and initial conditions $X_0 \sim \operatorname{Unif}[-7.5, 10]$. The invariant distributions are in this case $\mathcal{N}(\alpha,1)$, with $\alpha \approx -4.23$ (stable), $\alpha \approx -1.83$ (unstable), $\alpha \approx 1.35$ (stable), $\alpha \approx 5.38$ (unstable), and $\alpha \approx 6.88$ (stable), see Figure~\ref{fig:quentin_inv_exact_approx}. A histogram of a simulation of decoupled particles at time $T=5$ is also shown in Figure \ref{fig:quentin_inv_exact_approx}.

\begin{figure}
    \centering
    \includegraphics[width=0.6\linewidth]{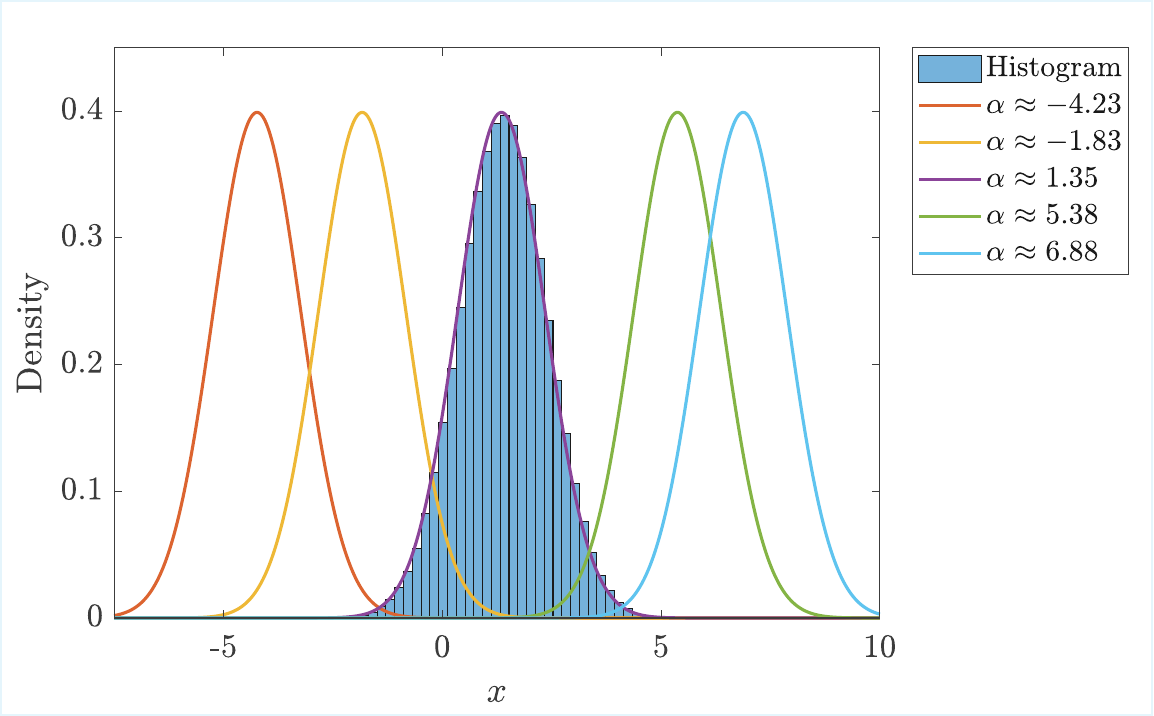}
    \caption{The exact invariant distributions of the Cormier model are shown, which are of the form $\mathcal{N}(\alpha,1)$ with the parameters $\alpha$ mentioned in the legend. The histogram shows the approximate distribution of 500\ts000 numerically simulated particles at $T=5$, with initial distribution $\operatorname{Unif}[-7.5,10]$. This shows that the particles converge to one of the stable invariant distributions.}
    \label{fig:quentin_inv_exact_approx}
\end{figure}

It is evident that all the particles are interacting and converging to one of the stable invariant distributions, more precisely they converge to the invariant distribution with $\alpha \approx 1.35$. For different initial conditions the particles can converge to a different stable invariant distribution.

In order to investigate the metastability of this system, we apply the EDMD algorithm using the lag-time $T=0.5$. Since the associated linear system behaves like an Ornstein--Uhlenbeck process, we expect the eigenvalues and eigenfunctions of the Cormier model to match those of the Ornstein--Uhlenbeck process. More specifically, the eigenfunctions of the Koopman operator should be the Hermite polynomials and the eigenvalues should be $\lambda_j=e^{-(j-1)T}$ where $ j = 1, 2, \dots $, with $j=1$ corresponding to the first eigenvalue and so on. This expected behavior is corroborated by the numerical experiments illustrated in Figures~\ref{fig:quentin_p_eig_approx} and \ref{fig:quentin_k_eig_approx},
and also by obtaining that the computed eigenvalues are $\lambda_1\approx1$, $\lambda_2\approx0.6$, $\lambda_3\approx0.36$, and $\lambda_4\approx0.22$. The basis functions used are the indicator functions (with the eigenfunction plots showing the linear interpolation between them). The example illustrates that the proposed method correctly identifies the eigenvalues and eigenfunctions of the Koopman operator associated with the Cormier model.

\begin{figure}
    \centering
    \begin{subfigure}[t]{0.45\textwidth}
        \centering
        \caption{Perron--Frobenius}
        \includegraphics[height=2in]{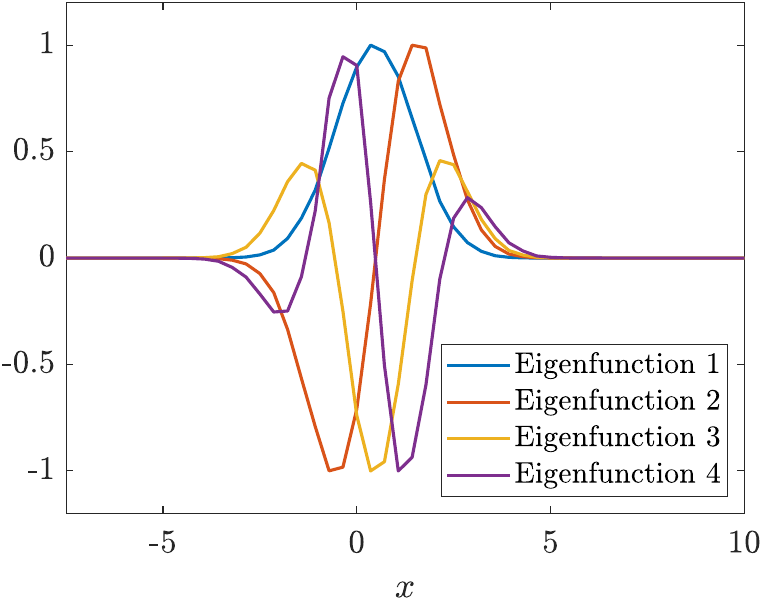}
        \label{fig:quentin_p_eig_approx}
    \end{subfigure}
    ~
    \begin{subfigure}[t]{0.45\textwidth}
        \centering
        \caption{Koopman}
        \includegraphics[height=2in]{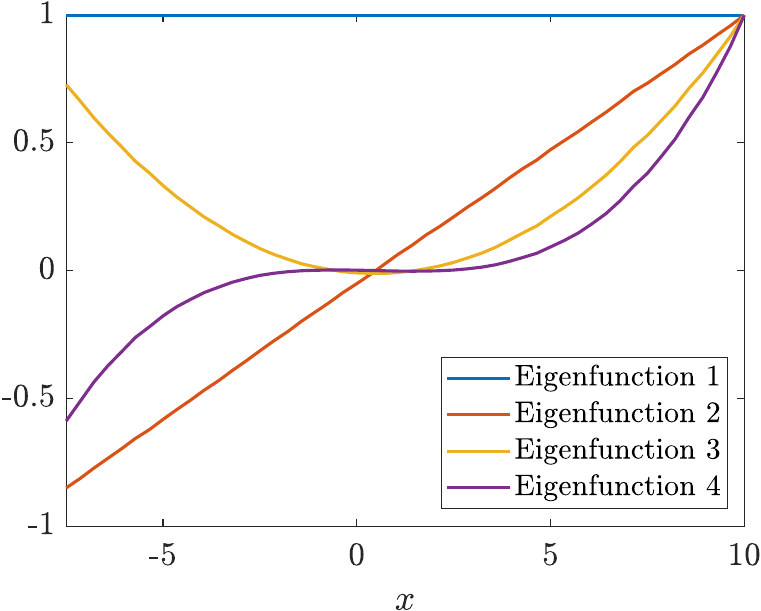}
        \label{fig:quentin_k_eig_approx}
    \end{subfigure}
    \caption{The plot (a) shows the eigenfunctions of the Perron--Frobenius operator associated with the Cormier model computed using EDMD, while (b) shows the eigenfunctions of the Koopman operator.}
\end{figure}

\subsection{Kuramoto on the circle}

The Kuramoto model on the circle, which plays an important role in neuroscience, has been previously studied as a mean-field SDE \cite{ANGELI2023, bertoli_2025, bicego2025computation} and has many variations. The variant we study here is the interacting particle system
\begin{equation*}
    \mathrm{d}X_t^m = \left(2 \sin(2 X_t^m) - \frac{1}{M} \sum_{j=1}^{M} \sin(X_t^m - X_t^j) \right) \mathrm{d}t + \sqrt{2\sigma} \ts \mathrm{d}W_t^m, \quad X_0^m \sim \xi^m,
\end{equation*}
for $ m = 1, 2, \dots, M $, which for $\sigma>0.7709$ has the unique invariant distribution $\rho = \frac{1}{Z} e^{-\cos(2x)/\sigma}$, where $ Z $ is a normalization constant. In what follows, assume that $\sigma=1$ is fixed. The resulting unique invariant distribution is shown in Figure~\ref{fig:kuramoto_inv_exact}.

\begin{figure}
    \centering
    \includegraphics[width=0.45\linewidth]{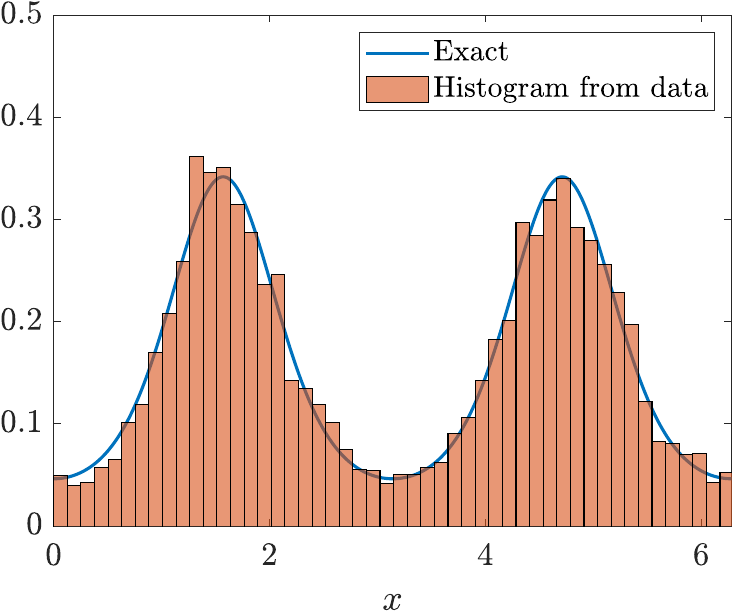}
    \caption{The exact invariant distribution of the Kuramoto mode on the circle is shown, as well as a histogram of 5000 numerically simulated particles at time $T=1$. This demonstrates that the particles converge to the exact invariant distribution.}
    \label{fig:kuramoto_inv_exact}
\end{figure}

The numerical experiments involve simulating the decoupled Euler--Maruyama scheme with 5000 particles, time step $h=0.01$, and final time $T=1$. A histogram of the final positions of the particles is shown in Figure~\ref{fig:kuramoto_inv_exact}, demonstrating that the particles converge to the exact invariant distribution.

The metastability of the system is illustrated with the aid of eigenvalues and eigenfunctions of the Koopman and Perron--Frobenius operators. The computed eigenvalues in decreasing order are $\lambda_{1}\approx 1 $, $ \lambda_{2}\approx 0.73 $, and $\lambda_{3}\approx 0.086$ (with $\lambda_n < \lambda_3$ for $n>3)$. Since all except the first two eigenvalues are small, there is only one form of metastability in our system. The eigenfunctions are illustrated in Figures~\ref{fig:kuramoto_p_eig_approx}
and \ref{fig:kuramoto_k_eig_approx}. The eigenfunctions show that the metastable sets are $[0, \pi]$ and $[\pi, 2 \pi]$. The basis functions used here to implement EDMD are monomials up to order 7.

\begin{figure}
    \centering
    \begin{subfigure}[t]{0.45\textwidth}
        \centering
        \caption{Perron--Frobenius}
        \includegraphics[height=2in]{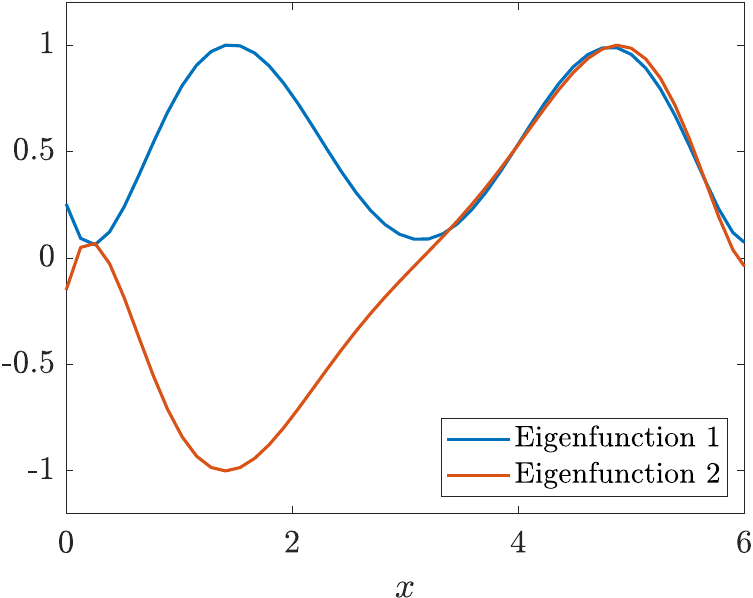}
        \label{fig:kuramoto_p_eig_approx}
    \end{subfigure}
    ~
    \begin{subfigure}[t]{0.45\textwidth}
        \centering
        \caption{Koopman}
        \includegraphics[height=2in]{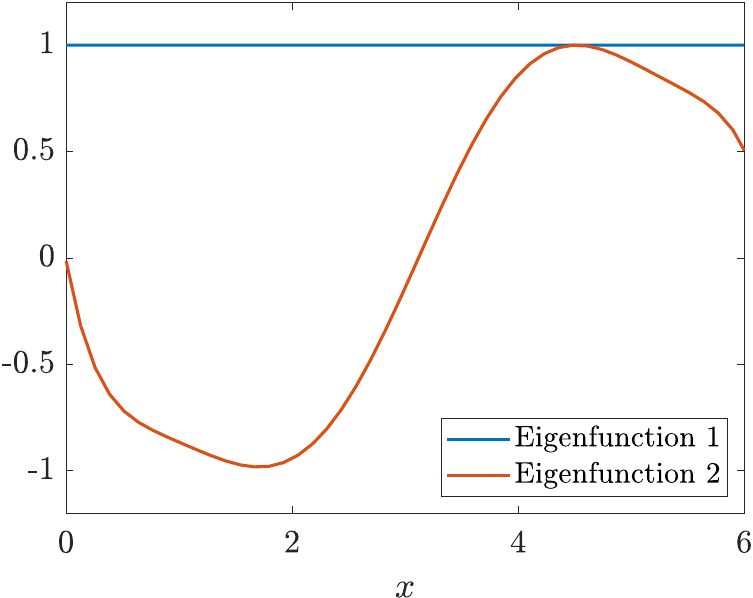}
        \label{fig:kuramoto_k_eig_approx}
    \end{subfigure}
    \caption{The plot (a) shows the eigenfunctions of the Perron--Frobenius operator for the Kuramoto model on the circle, and plot (b) shows the eigenfunctions of the Koopman operator. The data was duplicated and shifted by $\pi$ to enforce the symmetry and to improve the accuracy of the eigenfunctions.}
\end{figure}

Additionally, the metastable behavior can be further verified in the plots in Figure \ref{fig:kuramoto_traj_circle} after the particles of each metastable set are colored in blue and red, respectively. The figures show the particles on a circle at different times. It can be seen that any trajectories between $[0,\pi]$ tend to stay within that neighborhood at first, and the same holds for any particle starting in the region $[\pi, 2\pi]$. Therefore, it is clear that using the EDMD methodology one can split the space accurately into two metastable sets.

\begin{figure}
    \centering
    \begin{subfigure}[t]{0.49\textwidth}
        \centering
        \caption{t=0}
        \includegraphics[height=2in]{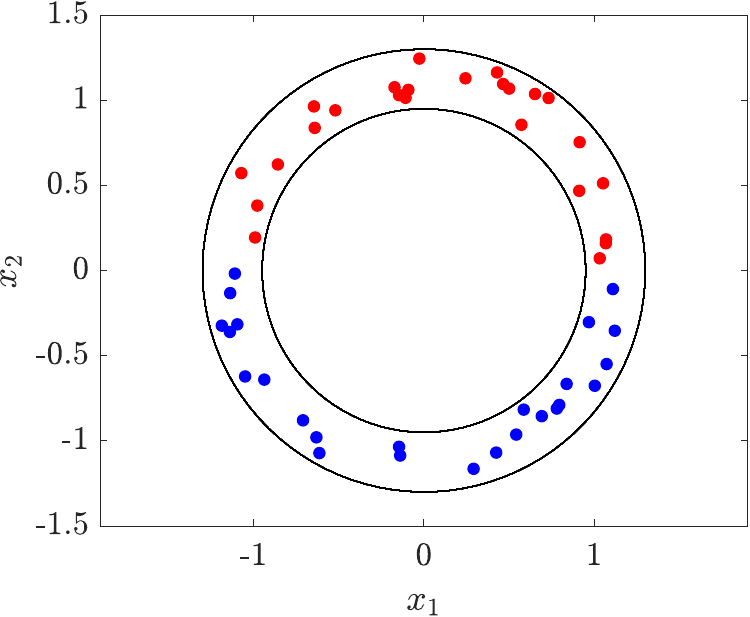}
    \end{subfigure}
    \begin{subfigure}[t]{0.49\textwidth}
        \centering
        \caption{t=0.2}
        \includegraphics[height=2in]{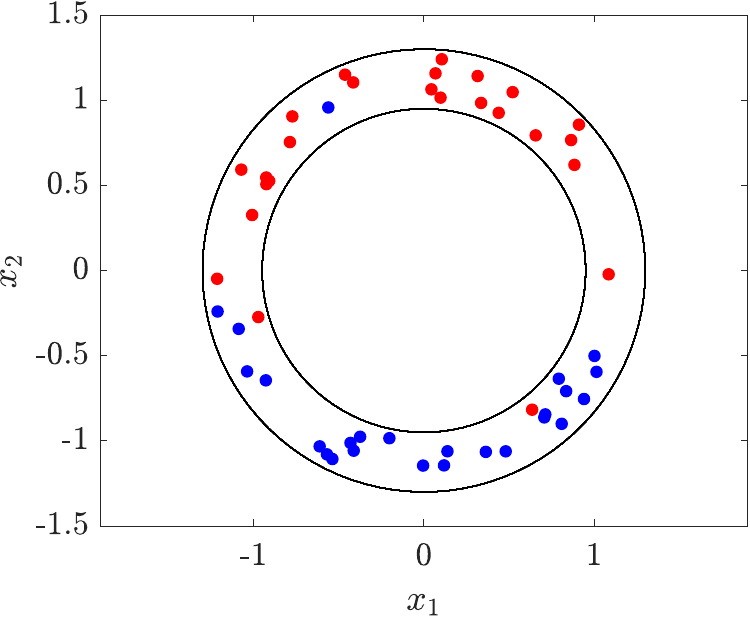}
    \end{subfigure} \\
    \begin{subfigure}[t]{0.49\textwidth}
        \centering
        \caption{t=0.6}
        \includegraphics[height=2in]{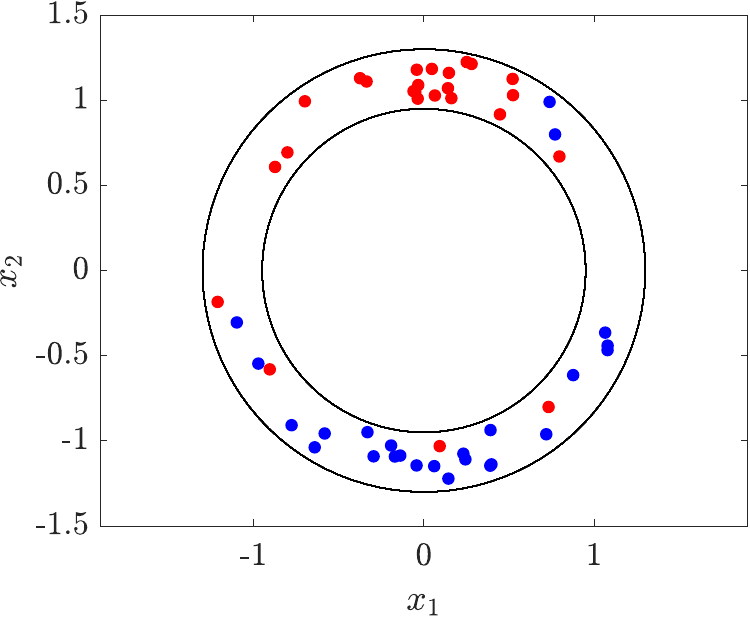}
    \end{subfigure}
    \begin{subfigure}[t]{0.49\textwidth}
        \centering
        \caption{t=1}
        \includegraphics[height=2in]{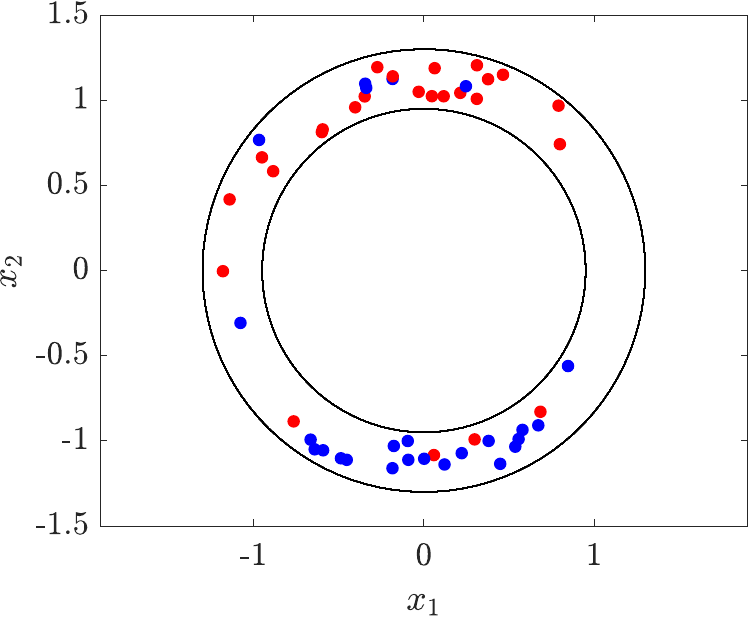}
    \end{subfigure}
    \caption{The scatter plots of positions of particles at times $t=0$, $t=0.2$, $t=0.6$, and $t=1$ are shown for 50 particles (out of the 5000 particles of the simulated system). Particles which at time $t=0$ are in the set $[0,\pi]$ are coloured red, whilst particles starting in $[\pi, 2 \pi]$ are coloured blue.}
    \label{fig:kuramoto_traj_circle}
\end{figure}

\subsection{Kuramoto on the sphere}

The model we consider here, inspired by \cite{kuramoto_sphere_2021}, is a modified version of the interacting particle system called the Kuramoto model on a sphere. The original system is deterministic, but we modify it by adding Brownian motion. The modification is based on techniques described in \cite{sdes_sphere_2016} that ensure that the particles remain on the sphere. Our interacting particle system is given by
\begin{align*}
     \mathrm{d}X_t^m &= \qty((A-\gamma^2 I) X_t^m + (I - X_t^m (X_t^m)^\top) \left(\frac{\alpha}{M}\sum_{j=1}^M X_t^j + \beta\mathbf{1}\right)) \ts \mathrm{d}t + \qty(I - X_t^m (X_t^m)^\top) \gamma \ts \mathrm{d}W_t^m, \\ \quad &X_0^m \sim \operatorname{Unif}(S^2),~ m = 1, 2, \dots, M,
\end{align*}
where $ S^2 $ denotes the unit sphere in $ \R^3 $ and $\mathbf{1}$ is a three dimensional vector of ones. We choose $\beta = -20$ with probability $\frac{1}{2}$ and $\beta=20$ with probability $\frac{1}{2}$ (i.e., $\beta = 40(b - \frac{1}{2})$ for $b \sim \operatorname{Bernoulli}(\frac{1}{2})$), $A$ is an antisymmetric matrix in $\mathbb{R}^{3 \times 3}$, $\gamma$ and $\alpha$ are positive constants. It holds that $X_t^m \in \mathbb{R}^3$ and more precisely $\norm{X_t^m }_2=1$ for all $m=1,2,...,M$, since the particles move on a sphere. As $M \to \infty$ the randomness of $\beta$ remains as is, which yields the following mean-field SDE with random drift coefficient
\begin{equation*}
    \mathrm{d}X_t = \qty((A-\gamma^2 I) X_t + (I - X_t (X_t)^\top) (\alpha \E[X_t] + \beta\mathbf{1})) \ts \mathrm{d}t + \qty(I - X_t (X_t)^\top) \gamma \ts \mathrm{d}W_t,
\end{equation*}
and $X_0$ is uniformly distributed on the sphere.

The numerical simulation of this model is done by implementing the decoupled Euler--Maruyama scheme with projection onto the sphere (as the discretization errors accumulate and particles tend to deviate away from the sphere). The parameters used are $\gamma=0.5$, $\alpha=0.5$, dimension $d=3$, time step $h = 0.01$ and number of particles $M=5000$. The antisymmetric matrix $A \in \mathbb{R}^{3 \times 3}$ is randomly generated with non-zero entries uniformly generated in $(-1,1)$. The basis functions used are based on a Voronoi discretization of the sphere. More specifically, the basis function $(\psi_i)_{i = 1, 2, \dots, N}$ are indicator functions, with $\psi_i$ taking the value $1$ for all points in region $i$, and taking the value $0$ everywhere else. In our simulations $N=200$ basis functions were used.

\begin{figure}
    \centering
    \begin{subfigure}[t]{0.45\textwidth}
        \centering
        \caption{}
        \includegraphics[height=2in]{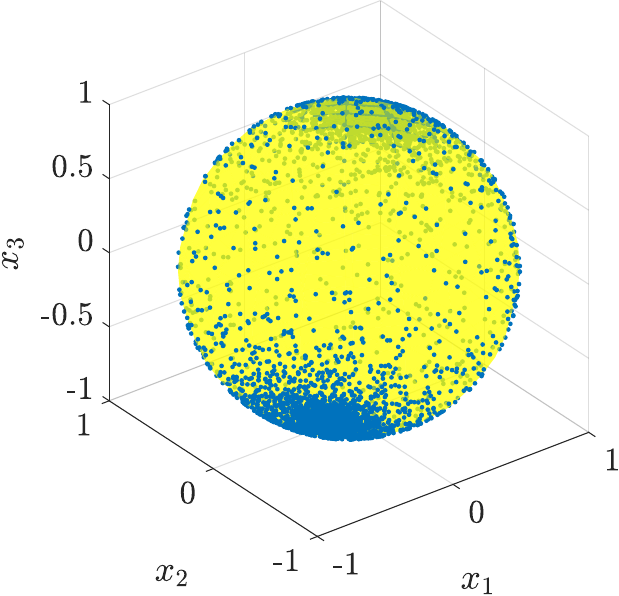}
    \end{subfigure}
    \begin{subfigure}[t]{0.45\textwidth}
        \centering
        \caption{}
        \includegraphics[height=2in]{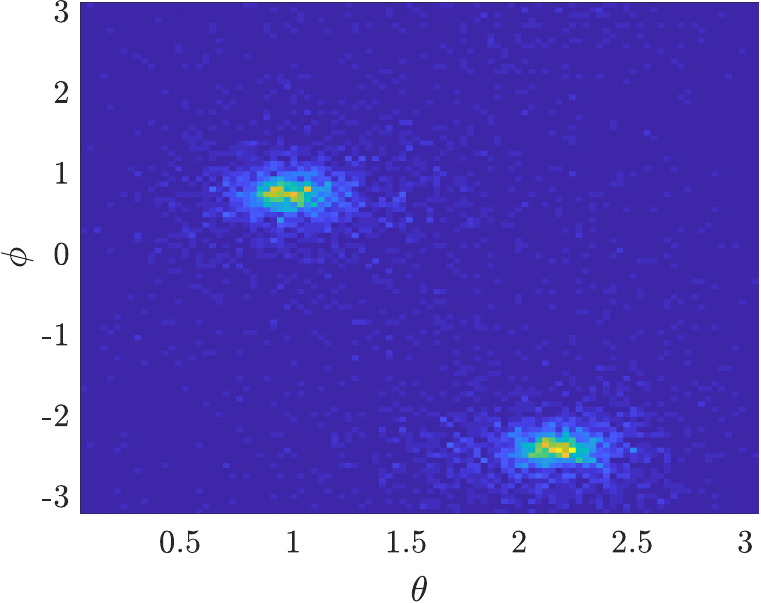}
    \end{subfigure}
    \caption{The first plot (a) shows the positions of 5000 numerically simulated particles of the Kuramoto on the sphere model at time $T=3$. The plot (b) is the histogram of the azimuthal angles $\phi$ and polar angles $\theta$ of the positions of 5000 particles at time $T=3$.}
\label{fig:kuramoto_sphere_inv_approx}
\end{figure}

The approximated invariant distribution of the particles can be seen in Figure \ref{fig:kuramoto_sphere_inv_approx}, which depicts the positions of all the 5000 particles at time $T=3$. The two dominant Koopman and Perron--Frobenius eigenfunctions are depicted in Figures~\ref{fig:kuramoto_sphere_eig}, with the basis functions values being linearly interpolated for the sake of visualization. The lag time used for estimating eigenfunctions was $T=0.5$. The eigenvalues obtained for this lag time are $\lambda_1 \approx 1$, $\lambda_2 \approx 0.662$, $\lambda_3 \approx 0.215$, and $\lambda_4 \approx 0.215$. The first eigenfunction of Koopman operator is the constant function as expected. The second eigenfunction has positive values in half of the sphere and negative values in the other half of the sphere. This indicates that the model has metastability between the two hemispheres, more precisely, particles in one hemisphere tend to stay in that hemisphere, and particles in the other hemisphere tend to stay in that region. This metastability can also be seen through the eigenfunctions of the Perron--Frobenius operator. The first eigenfunction of the Perron--Frobenius operator is indicating the invariant distribution, where there are particles concentrated mostly at the two poles of the sphere. The second eigenfunction indicates that one of the poles of the sphere is positive and one is negative, which further verifies the metastability of the model.

\begin{figure}
    \centering
    \begin{subfigure}[t]{0.45\textwidth}
        \centering
        \caption{Eigenfunction 1}
        \includegraphics[height=2in]{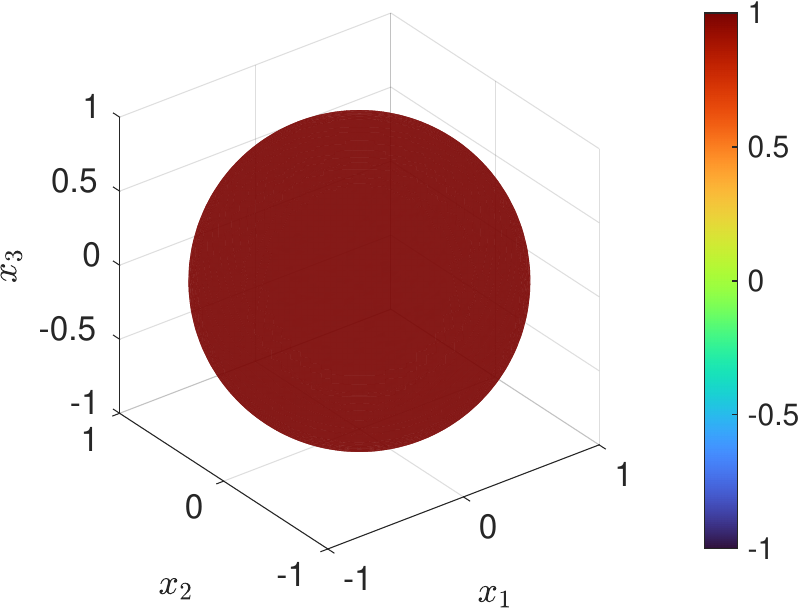}
    \end{subfigure}
    ~
    \begin{subfigure}[t]{0.45\textwidth}
        \centering
        \caption{Eigenfunction 2}
        \includegraphics[height=2in]{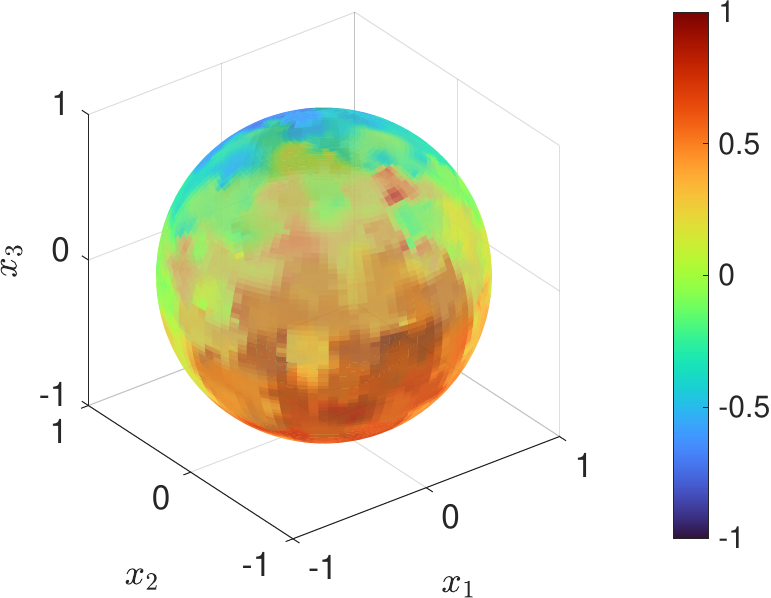}
    \end{subfigure}
    \begin{subfigure}[t]{0.45\textwidth}
        \centering
        \caption{Eigenfunction 1}
        \includegraphics[height=2in]{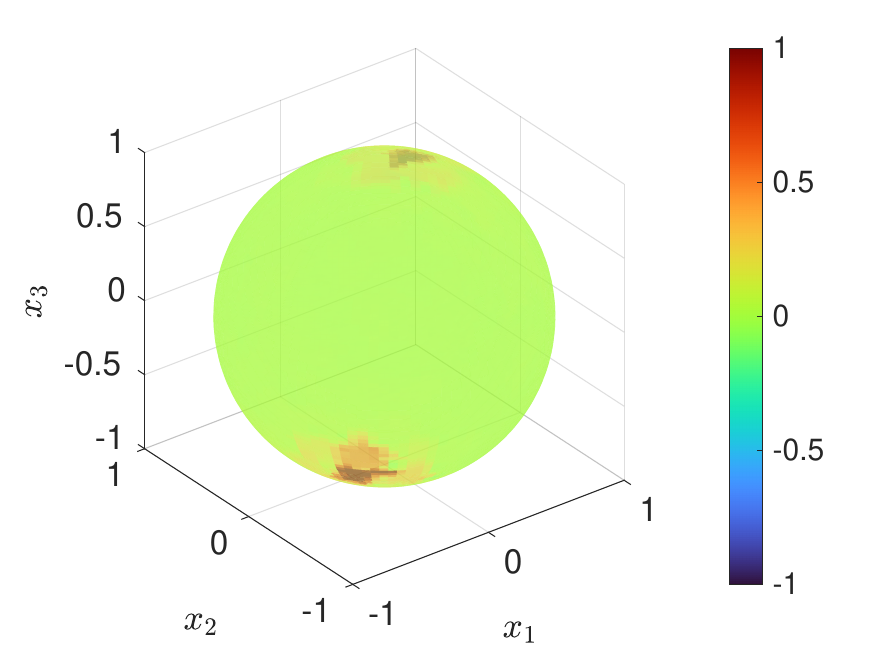}
    \end{subfigure}
    ~
    \begin{subfigure}[t]{0.45\textwidth}
        \centering
        \caption{Eigenfunction 2}
        \includegraphics[height=2in]{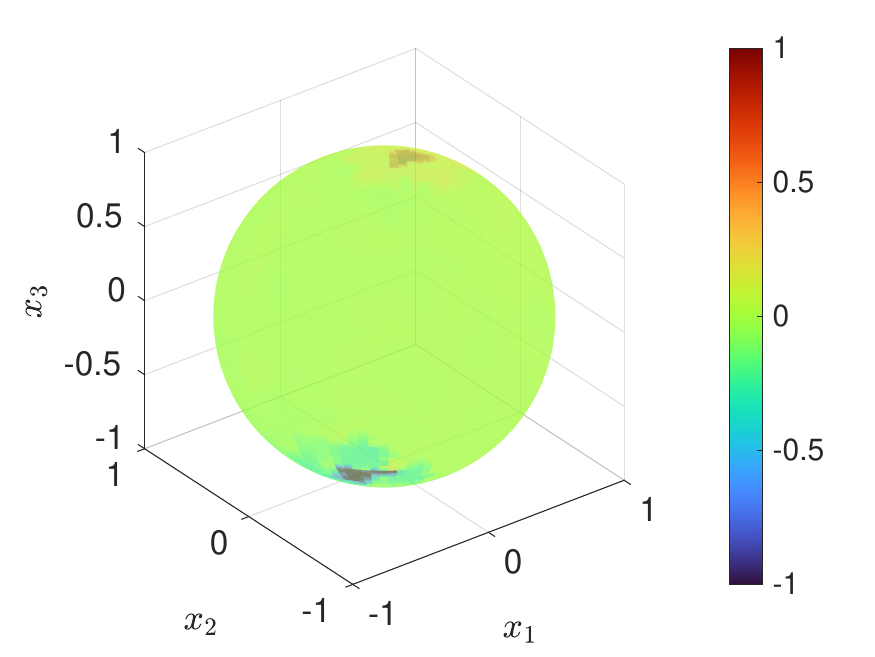}
    \end{subfigure}
    \caption{The (a) first and (b) second eigenfunctions, computed using EDMD, of the Koopman operator associated with the Kuramoto model on the sphere. The (c) first and (d) second eigenfunctions of the Perron--Frobenius operator.}
    \label{fig:kuramoto_sphere_eig}
\end{figure}

\section{Conclusion and open questions}
\label{sec:conclusion}

In this work, we have considered the behavior of interacting particle systems, with absolute symmetry of permutation of particles, when the number of interacting particles tends to infinity. These limiting equations are called McKean--Vlasov equations, and we have attempted to understand their global behavior. We have shown that it is possible to infer important characteristic properties of these stochastic dynamical systems using associated transfer operators such as the Koopman operator and the Perron--Frobenius operator. In order to circumvent the problem of the Koopman operator being nonlinear in the case of the McKean--Vlasov SDE, we have worked with the decoupled McKean--Vlasov SDE, which ensured that the transfer operators are well-defined and their eigenvalues and eigenfunctions are meaningful.
We have estimated the transfer operators by first projecting them onto a finite-dimensional subspace via Galerkin projection and using numerically simulated data to approximate the required integrals. We have shown both theoretically and numerically that these data-driven Galerkin projections are appropriate tools for obtaining important information about the underlying stochastic mean-field system.

The work presented in this paper can be extended in several interesting directions. For instance, one could exploit the reproducing kernel Hilbert space methodology to infer information about mean-field stochastic systems (this has been done in the non-stochastic case \cite{RKHS_2023}) or one could apply the gEDMD algorithm \cite{KNPNCS20} using the decoupled McKean--Vlasov SDE. Additionally, a second objective could be to investigate how effective this methodology is when it comes to finding clusters of interacting particle systems. Another opportunity for future work would be to investigate the accuracy of the estimated transfer operators depending on the step size and number of particles used. Additionally, one could analyze the effect of using finitely many dictionary basis functions on the spectra of the Koopman operator. We will return to these problems in future work.

\section*{Acknowledgments}

We would like to thank Liam Llamazares-Elias and Samir Llamazares-Elias for many interesting discussions about error bounds and convergence rates. E.I. was supported by the EPSRC Centre for Doctoral Training in Mathematical Modelling, Analysis and Computation (MAC-MIGS) funded by the UK Engineering and Physical Sciences Research Council (grant EP/S023291/1), Heriot--Watt University and the University of Edinburgh. G.d.R.~acknowledges support from the Funda{\c c}\~ao para a Ci\^{e}ncia e a Tecnologia (Portuguese Foundation for Science and Technology) through the projects \href{https://doi.org/10.54499/UIDB/00297/2020}{UIDB/00297/2020} and \href{https://doi.org/10.54499/UIDP/00297/2020}{UIDP/00297/2020} (Center for Mathematics and Applications, NOVA Math) and by the UK Research and Innovation (UKRI) under the UK government’s Horizon Europe funding Guarantee [Project UKRI343].

\bibliographystyle{unsrturl}
\bibliography{bibliography}

\appendix

\section{Auxiliary results}

\begin{lemma}[Law of large numbers {\cite[IV.\textsection3]{Shiryaev1996BookProbability}}]
\label{obs_l2_wlln}
Let $\{X_j\}_{j \geq 1}$ be a sequence of independent random variables with $\E[X_j] = \mu$ and define $S_n=X_1+X_2+ \dots +X_n$.
\begin{enumerate}[leftmargin=3.5ex, itemsep=0ex, topsep=0.5ex, label=\roman*)]
\item If $\sup_j \Var(X_j) \leq C < \infty$, then it holds that
\begin{equation*}
    \frac{S_n}{n} \xrightarrow{} \mu  \quad \text{as }  n \to \infty \quad \text{in $L^2(\Omega)$.}
\end{equation*}
\item If the family $\{X_j\}_{j \geq 1}$ is i.i.d.\ with $\bE\big[\abs{X_1}\big]<\infty$, then $ \frac{S_n}{n} \to \mu$ as $n \to \infty$ almost surely.
\end{enumerate}
\end{lemma}

\section{Proof of results for numerical scheme for the decoupled McKean--Vlasov SDE}
\label{appendix:num_scheme_results}
Here, we prove the necessary results needed for the convergence of the numerical scheme for the decoupled McKean--Vlasov equation defined in Subsection \ref{subsec:numerical_results}.

The decoupled equation \eqref{eq:decoupledMVSDE}, once $\mu$ is fixed, is a standard SDE whose coefficients have an added time dependency arising from the law $\mu_t$. Hence, standard SDE results for well-posedness and uniqueness of solution, which can be found in \cite[Chapter 3]{ZhangBook_2017}, hold for equation \eqref{eq:decoupledMVSDE} if $(t,x)\mapsto b(t,x,\mu_t)$ and $(t,x)\mapsto \sigma(t,x,\mu_t)$ are uniformly Lipschitz in space and $\nicefrac{1}{2}$-H\"older in time. The latter is given by Assumption \ref{assumption:b_sigma} whereas the former is proven in Lemma~\ref{lem:TimeRegularityofDecoupledCoefficients-Holder} below.

\begin{lemma}
\label{lem:TimeRegularityofDecoupledCoefficients-Holder}
Suppose $b,\sigma$ satisfy Assumption \ref{assumption:b_sigma}.
Let $(\mu_t)_t$ be as given by \eqref{eq:MVSDE}.
Then, for any $x\in\bR^d$ the maps $t \mapsto b(t,x,\mu_t)$ and $t \mapsto \sigma(t,x,\mu_t)$ are $\nicefrac{1}{2}$-H\"older continuous in $t$ (uniformly in space).
\begin{proof}
The time regularity of $t\mapsto a(t,\cdot)=b(t,\cdot,\mu_t)$ can be established from the assumed $\nicefrac{1}{2}$-H\"older regularity in time $b(t,0,\delta_0)$ and the $\nicefrac{1}{2}$-H\"older regularity in time of $t\mapsto \mu_t$ is established in Theorem~2.3 in \cite{chen2024IMA} and Section 3.2 of \cite{salkeld2019LDP-MV-SDES}.
We have $| b(t, \cdot, \mu_t)-b(s, \cdot, \mu_s)| \leq C |t-s|^{1/2}$. The result for $\sigma$ follows in the same way.
\end{proof}
\end{lemma}

\begin{proposition}
\label{proposition_decoupled_ineq1}
Let Assumption \ref{assumption:b_sigma} (with $L_b=L_{\sigma}=L$) hold. Assume a probability measure $\hat \mu_t \in \mathcal{Q}(\mathcal{Q}_2(\bR^d))$ is available at any $t\in[0,T]$ and is adapted to the filtration. Let the solution of the decoupled McKean--Vlasov SDE \eqref{eq:decoupledMVSDE} be $X^{x,\hat{\mu}}$.
Then, there exists some constant $C>0$ dependent on time $T$ and Lipschitz constant $L$ (and independent of the initial condition point $x$) such that
\begin{align}
\label{eq:momentboundDecoupledMVSDE-input-hat-mu}
    \bE\big[ \sup_{t \in [0,T]} |X^{x,\hat{\mu}}_t|^2 \big ]
    & \leq
    C \Big(1+|x|^2 + \sup_{t \in [0,T]}\bE\big[\,|W_2(\hat \mu_t,\delta_x)|^2 \big] \Big) e^{CT},
\end{align}
and
\begin{align}
\label{eq:momentboundDecoupledMVSDE-input-hat-mu-stabilitybound}
\norm{\xmu - \xmuhat}_{L^2(\Omega)}^2
\leq
C \sup_{ s \in[0,T]} \E[\,|W_2(\mu_s,\hat{\mu}_s)|^2],
\end{align}
where $\xmu$ is the solution to \eqref{eq:decoupledMVSDE}.
\begin{proof}
Once $\hat \mu$ is fixed, the decoupled McKean--Vlasov equation is just a standard SDE and thus this proof is a general result from standard SDE theory \cite[Theorem 3.2.4]{ZhangBook_2017}. The second statement can be proven by taking the expectation of the difference $|\xmu - \xmuhat|^2$, using $|a+b|^2\leq a^2 + b^2$, and using It\^o's isometry to get
\begin{align*}
\norm{\xmu - \xmuhat}_{L^2(\Omega)}^2
&\leq 2 T \qty(2 L^2 T \sup_{s\in[0,T]} \E[W_2^2(\mu_s,\hat{\mu}_s)] + 2L^2 \int_0^T \E[|X_s^{x,\mu}-X_s^{x,\hat{\mu}}|^2] \,ds)
 \\
& \quad + 2 \qty(2 L^2 T \sup_{s\in[0,T]} \E[W_2^2(\mu_s,\hat{\mu}_s)] + 2L^2 \int_0^T \E[|X_s^{x,\mu}-X_s^{x,\hat{\mu}}|^2] \,ds).
\end{align*}

Then applying Gronwall's inequality gives us the desired result with $C$ being $C:=(4L^2T^2 + 4L^2T)e^{4L^2T^2+4L^2T} $.
\end{proof}
 \end{proposition}

For $T>0$ and over the time interval $[0,T]$ define the uniform timegrid $\pi^{h}=\{t_k= k h:k=0,\ldots,K\}$ where $h = {T}/{K}<1$ for some $K\in \mathbb{Z}^+$. The following corollary proves the convergence of the non-implementable Euler scheme.

\begin{corollary}
\label{coro:Conv-of-non-implem-Euler}
Let $(X^{x,\mu}_t)_{t\in[0,T]}$ be the solution to \eqref{eq:decoupledMVSDE}.
Then, the non-implementable Euler scheme $\{X^{x,\mu,h}_{t_{k}}\}_{t_k \in \pi^h}$ \eqref{eq:non-implementable-EM} satisfies
\begin{align*}
    \max_{k=0,\ldots, |\pi|-1}
    \bE \big[\, \sup_{t \in [t_k, t_{k+1}]} | X^{x,\mu}_t - X^{x,\mu,h}_{t_{k}}|^2\big] \leq C(1+|x|^2)h.
\end{align*}
\begin{proof}
The result follows by using Lemma \ref{lem:TimeRegularityofDecoupledCoefficients-Holder} to verify the conditions of Theorem in \cite[Chapter 3]{ZhangBook_2017} applied to the decoupled McKean--Vlasov SDE \eqref{eq:decoupledMVSDE}.
\end{proof}
\end{corollary}

\begin{proposition}
\label{prop_compare_num_scheme_decoupled}
Let Assumption \ref{assumption:b_sigma} hold.
Set $\{X^{x,\mu,h}_{t_{k}}\}_{t_k \in \pi^h}$ as the non-implementable Euler scheme approximation \eqref{eq:non-implementable-EM}.
Assume a computable family of probability measures, denoted $\{\hat \mu_{t_k}\}_{t_k\in\pi^{h}}$, over $\mathcal{Q}_2(\bR^d)$ or $\mathcal{Q}(\mathcal{Q}_2(\bR^d))$ (i.e., deterministic or random and adapted), is available as an approximation to $\{\mu_{t_k}\}_{t_k\in\pi^h}$ and introduce $\{X_{t_k}^{x,\hat{\mu},h}\}_{t_k \in \pi^h}$ as the implementable Euler scheme approximation given by \eqref{eq:implementable-EM}.
Then, for some positive constant $C_d$ dependent on $d,T, L_b$ and $L_{\sigma}$ we have if $\hat{\mu}$ is deterministic
\begin{align}\label{eq:bound_by_w}
\norm{ X^{x,\mu,h}_{t_{K}}-X_{t_K}^{x,\hat{\mu},h} }_{L^2(\Omega)}^2
\leq
   C_d
   \sup_{k \in \{0,1,2,...,K\}} W_2(\mu_{t_k},\hat{\mu}_{t_k})^2
\end{align}
or if $\hat{\mu}$ is random
\begin{align}\label{eq:bound_by_e_w}
\norm{ X^{x,\mu,h}_{t_{K}}-X_{t_K}^{x,\hat{\mu},h} }_{L^2(\Omega)}^2
\leq
   C_d
   \sup_{k \in \{0,1,2,...,K\}} \E[W_2(\mu_{t_k},\hat{\mu}_{t_k})^2].
\end{align}

\begin{proof}
Let $k\in\{0,\ldots,K-1\}$. Define the auxiliary quantity $e_{k+1}:=|X^{x,\mu,h}_{t_{k+1}}-X_{t_{k+1}}^{x,\hat{\mu},h}|$ with $e_{0}=0$.

Taking the difference between the two schemes in absolute value, squaring both sides, then taking expectation on both sides (and using the tower property), and then using the Lipschitz and H{\"o}lder conditions given by Assumption \ref{assumption:b_sigma} (and using $2ab<a^2 + b^2$) results in the following inequality
\begin{align*}
      \E[e_{k+1}^2]
      &
      \leq
      \E[e_k^2] + 2 L_b^2 \E[e_k^2](h)^2 + 2 L_b^2 W_2(\mu_{t_k},\hat{\mu}_{t_k})^2 (h)^2 + 2 L_\sigma^2 \E[e_k^2] h + 2 L_{\sigma}^2 W_2(\mu_{t_k},\hat{\mu}_{t_k})^2 h \\ & + 2 L_b \E[e_k^2] h + L_b^2 \E[e_k^2] h + W_2(\mu_{t_k},\hat{\mu}_{t_k})^2 h
      \\ & \leq \E[e_k^2] \qty(1 + 2 L_b^2 h^2 + 2 L_\sigma^2 h + 2 L_b h + L_b^2 h ) + W_2(\mu_{t_k},\hat{\mu}_{t_k})^2 \qty(2 L_b^2 h^2 + 2 L_\sigma^2 h+ h)
       \\ & \leq \E[e_k^2] \qty(1 + a h + d h^2 ) + w_k,
    \end{align*}
where $a = 2L_\sigma^2 + 2 L_b + L_b^2$, $d= 2 L_b^2$ and $w_k=W_2(\mu_{t_k},\hat{\mu}_{t_k})^2 \qty(2 L_b^2 h^2 + 2 L_\sigma^2 h+ h)$. Now let $A=1 + a h + d h^2 $, by recursion we have
\begin{align*}
\E[e_{k}^2] &
      \leq A^k \E[e_{0}^2] + A^{k-1} w_0 + A^{k-2} w_1 + ... + A^2 w_{k-3} + A w_{k-2} + w_{k-1}.
\end{align*}
Let $w:= \sup_{k \in \{0,1,2,...,K\}} W_2(\mu_{t_k},\hat{\mu}_{t_k})^2\qty(2 L_b^2 h^2 + 2 L_\sigma^2 h+ h)$. Using that $\E[e_0^2]=\E[|\hat{X}_{0}-X_{0}^{x,\hat{\mu},h}|^2]=\E[|x-x|^2]=0$ we can rewrite the expression above as
\begin{align*}
\E[e_{K}^2] &
\leq \frac{A^K - 1}{A-1} w.
\end{align*}
In order to bound $A^K$ (for all $K$) we use the inequality that for $x>0$, it holds that $(1+x)^r < e^{rx}$ for all $r \in \mathbb{Z}^+$. Consider
\begin{align*}
    A^K & = \Big(1 + a h + d h^2\Big)^K
    \\ & = \Big(1 + a \frac{T}{K} + d \frac{T^2}{K^2}\Big)^K
    \leq e^{aK \frac{T}{K} + dK \frac{T^2}{K^2}}
    \leq e^{aT + dT^2 \frac{1}{K}}
    \leq e^{aT+dT} \quad \text{since  } K > T.
\end{align*}
We need to also consider the quantity $\frac{w}{A-1}$ as $h \rightarrow 0$
\begin{align*}
    \frac{w}{A-1} &
    \leq \sup_{k \in \{0,1,2,...,K\}}W_2(\mu_{t_k},\hat{\mu}_{t_k})^2 \frac{2 L_b^2 + 2 L_\sigma^2 + 1 }{ 2 L_\sigma^2 + 2 L_b + L_b^2 }
     \quad \text{since  } 0<h<1.
\end{align*}
Therefore, putting everything together
\[\E[e_{K}^2] \leq e^{(2L_\sigma^2 + 2 L_b + 3 L_b^2)T} \qty(\frac{2 L_b^2 + 2 L_\sigma^2 + 1 }{ 2 L_\sigma^2 + 2 L_b + L_b^2 } ) \sup_{k \in \{0,1,2,...,K\}} W_2(\mu_{t_k},\hat{\mu}_{t_k})^2 \]
for all $K$.

The proof of \eqref{eq:bound_by_e_w} is a straightforward adaptation of the proof of \eqref{eq:bound_by_w} by taking expectation on both sides, with the expectation being taken over the randomness of $\hat{\mu}$.
\end{proof}
\end{proposition}

In order to bound the quantity $\E\qty[ \sup_{t \in [0,T]\}} W_2(\mu_{t},\hat{\mu}_{t})^2]$ when $\hat{\mu}$ is as in \eqref{eq:mu_hat} we need the following results.

\begin{proposition}[Proposition 2.5 in \cite{chen2024IMA}]
	\label{Prop:Propagation of Chaos}
    Let $\bS^j([0,T])$ for $j \geq  1$, be the space of processes $\bR^d$ valued, $\mathcal{F}-$adapted $Z$, which satisfy $(\E[ \sup_{t \in [0,T]} |Z_t|^{j}])^{1/j} < \infty$.
    Let Assumption \ref{assumption:b_sigma} hold and take $\tilde{M}\in \bN$ from \eqref{eq:IPS}. Then, for any $p\geq 2$ there exists a unique solution $\{X^{m}\}_m$ to \eqref{eq:IPS} in $\bS^p([0,T])$.
Also, then, there exists a constant $C>0$ independent of $\tilde{M}$ (but depending on $T$ and $p$) such that
\begin{align}
\label{eq:poc result - MVSDE-2-Empirical-of-nIPS}
\sup_{t\in [0,T]} \E\big[W_2( \mu_t, \mu^{\tilde{M}}_t
  )^2\big]
\le C
    \begin{Bmatrix}
        \tilde{M}^{-1 / 2}, & d<4 \\
    \tilde{M}^{-1 / 2} \log \tilde{M}, & d=4 \\
    \tilde{M}^{\frac{-2}{d}}, & d>4
    \end{Bmatrix}
    := g(\tilde{M}),
\end{align}
where $\mu^{\tilde{M}}_t:= \frac{1}{\tilde{M}} \sum_{m=1}^{\tilde{M}} \delta_{X^m_t}$ is the empirical measure formed by the solutions $X^m$ of the IPS system~\eqref{eq:IPS}.
\end{proposition}

\begin{theorem} [Mean-square convergence]
\label{theorem:mean_square_ips}
Let Assumption \ref{assumption:b_sigma} hold and choose $\tilde{h}$ sufficiently small.
Let $m\in\{ 1,2,...,\tilde{M}\}$, take $X^{m}$ to be the solution to \eqref{eq:IPS} and let $X^{m,\tilde{h}}$ be the continuous-time extension of the Euler numerical scheme for the \eqref{eq:IPS} with time step $\tilde{h}$.
Then for the Euler scheme's empirical distribution $\mu^{\tilde{M},\tilde{h}}$ and the IPS's empirical measure $\mu^{\tilde{M}}$ in \eqref{eq:IPS} there exists a constant $C>0$ independent of $\tilde{h},\tilde{M}$ (but depending on $T$) such that
\begin{align}
\label{eq:convergence theroem term 1--ForEmpiricalMeasures}
     \sup_{ t \in [0,T] }
  \bE\Big[  \big(W_2(\mu^{\tilde{M}}_t ,\mu_t^{\tilde{M},\tilde{h}}) \big)^2 \Big]
\leq
     \sup_{t \in [0,T]}
     \frac{1}{\tilde{M}} \sum_{m=1}^{\tilde{M}}
    \bE\big[\,  |X_{t}^{m}-X_{t}^{m,\tilde{h}} |^2 \big]
  &
  \le C \tilde{h}.
\end{align}
\begin{proof}
This follows from \cite[Theorem 2.9]{chen2024IMA} after the usual domination of the Wasserstein distance by the $L^2$-norm of the difference of processes.
\end{proof}
\end{theorem}

By combining the above two results, we have
\begin{align*}
\nonumber
\E\qty[ \sup_{t \in [0,T]\}} W_2(\mu_{t},\hat{\mu}_{t})^2]
 &= \E\qty[ \sup_{t \in [0,T]} W_2(\mu_{t},\mu^{\tilde{M},\tilde{h}}_{t})^2]\\
    &\leq \sup_{t \in [0,T] }
\Big\{
\E[W_2(\mu_{t},\mu^{\tilde{M}}_{t} )^2]  +  \E[W_2(\mu^{\tilde{M}}_{t}, \mu^{\tilde{M},\tilde{h}}_{t})^2]
\Big\}
\leq g(\tilde{M}) + C\tilde{h}.
\end{align*}

\end{document}